\documentclass[a4paper, draft, oneside]{amsart}
\usepackage{amssymb}
\usepackage{amscd}
\usepackage{amsmath}
\usepackage[all]{xy}
\usepackage{mathrsfs}
 \usepackage{color}
 \usepackage[OT2,OT1]{fontenc}
  \usepackage[new]{old-arrows}

\theoremstyle{plain}
\newtheorem{theorem}{Theorem}[section]
\newtheorem{proposition}[theorem]{Proposition}
\newtheorem{lemma}[theorem]{Lemma}
\newtheorem{corollary}[theorem]{Corollary}

\theoremstyle{remark}
\newtheorem{remark}[theorem]{Remark}

\theoremstyle{definition}
\newtheorem{definition}[theorem]{Definition}

\newtheorem{conjecture}[theorem]{Conjecture}

  1

\DeclareMathOperator{\Gal}{Gal}

\DeclareMathOperator{\Hom}{Hom}

\DeclareSymbolFont{cyrletters}{OT2}{wncyr}{m}{n}
\DeclareMathSymbol{\Sha}{\mathalpha}{cyrletters}{"58}

\newcommand{\bF}{\mathbb{F}}
\newcommand{\bQ}{\mathbb{Q}}
\newcommand{\bL}{\mathbb{L}}

\newcommand{\bT}{\mathbb{T}}
\newcommand{\bZ}{\mathbb{Z}}

\newcommand{\cD}{\mathcal{D}}

\newcommand{\cF}{\mathcal{F}}
\newcommand{\cG}{\mathcal{G}}

\newcommand{\cM}{\mathcal{M}}
\newcommand{\cN}{\mathcal{N}}

\newcommand{\cP}{\mathcal{P}}

\newcommand{\cX}{\mathcal{X}}

\newcommand{\fl}{\mathfrak{l}}

\newcommand{\fp}{\mathfrak{p}}

\newcommand{\fm}{\mathfrak{m}}

\newcommand{\fL}{\mathfrak{L}}

\begin{document}

\title[]{$p$-Selmer group and Modular symbols}

\author{Ryotaro Sakamoto}

\begin{abstract} 
In this paper, we prove that the dimension of the $p$-Selmer group for an elliptic curve is controlled by certain analytic quantities associated with modular symbols, which is conjectured by Kurihara. 
\end{abstract}

\address{RIKEN Center for Advanced Intelligence Project\\Nihonbashi 1-chome Mitsui Building\\ 15th floor\\1-4-1 Nihonbashi\\Chuo-ku\\Tokyo
103-0027\\Japan}
\email{ryotaro.sakamoto@riken.jp}

\thanks{The author was supported by JSPS KAKENHI Grant Number 20J00456.} 

\maketitle

\tableofcontents


\section{Introduction}\label{sec:intro}

In modern number theory, it is an attractive area of research to connect $L$-values with  Selmer groups.  
In the present paper, we prove that the dimension of the (classical) $p$-Selmer group $\mathrm{Sel}(\bQ, E[p])$ for an elliptic curve $E/\bQ$ is controlled by certain analytic quantities associated with modular symbols, which is conjectured by Kurihara in \cite{Kur14b}. 

In order to explain this result in detail, we introduce some notations and hypotheses. 
Let $E/\bQ$ be an elliptic curve and 
let $S_{\rm bad}(E)$ denote the set of primes at which $E$ has bad reduction. 
For any integer $n \geq 0$, let $\bQ_{n}$ denote the $n$-th layer of the cyclotomic $\bZ_{p}$-extension of $\bQ$. 
As in the paper \cite{Kur14b} of Kurihara, we consider a prime $p \geq 3$ satisfying the following conditions: 
\begin{itemize}
\item[(a)] $p$ is a good ordinary prime for $E$. 
\item[(b)] The action of  $\Gal(\overline{\bQ}/\bQ)$ on $E[p]$ is surjective. 
\item[(c)]  $p \nmid  \# E(\bF_{p})\prod_{\ell \in S_{\rm bad}(E)}\mathrm{Tam}_{\ell}(E)$. 
\end{itemize}


Let $\cP_{1,0}$ denote the set of Kolyvagin primes, that is,  
\[
\cP_{1,0} := \{\ell \not\in S_{\rm bad}(E) \mid  E(\bF_{\ell})[p] \cong \bF_{p} \,  \textrm{ and } \, \ell \equiv 1 \pmod{p}\}. 
\]
We define $\cN_{1,0}$ to be the set of square-free products in $\cP_{1,0}$. 
We fix a generator $h_{\ell} \in \Gal(\bQ(\mu_{\ell})/\bQ)$ for each prime $\ell \in \cP_{1,0}$, and we obtain a  surjective homomorphism (induced by the discrete logarithm to the base $h_{\ell}$)  
\[
\overline{\log}_{h_{\ell}} \colon \Gal(\bQ(\mu_{\ell})/\bQ) \stackrel{\sim}{\longrightarrow} \bZ/(\ell-1) \longrightarrow \bF_{p}; \, h_{\ell}^{a} \mapsto a \bmod{p}. 
\]
Let $f_E$ denote the newform of weight $2$ associated with $E/\bQ$. 
Take an integer $d \in \cN_{1,0}$. 
For any integer $a$ with $(a,d) = 1$, 
we write $\sigma_{a} \in \Gal(\bQ(\mu_{d})/\bQ)$ for the element satisfying $\sigma_{a}(\zeta) = \zeta^{a}$ for any $\zeta \in \mu_{d}$ and put 
\[
[ a/d ] := 2 \pi \sqrt{-1} \int^{a/d}_{\sqrt{-1}\infty}f(z) \, \mathrm{d}z. 
\]
Following Kurihara in \cite{Kur14b}, we  define an analytic quantity $\widetilde{\delta}_{d}$ which relates to $L$-values by 
\[
\widetilde{\delta}_{d} := \sum_{\substack{a=1 \\ (a,d)=1}}^{d}\frac{\mathrm{Re}([a/d])}{\Omega_{E}^{+}} \cdot \prod_{\ell \mid d}\overline{\log}_{h_{\ell}}(\sigma_{a})  \in \bF_{p},   
\]
where $\Omega^{+}_{E}$ is the  N\'eron period of $E$. 
Kurihara remarked in \cite{Kur14b} that it is easy to compute the analytic quantity $\widetilde{\delta}_{d}$ (see \cite[\S 5.3]{Kur14b}), and gave the following conjecture.

\begin{conjecture}[{\cite[Conjecture 1]{Kur14b}}]\label{conj:1}
There is an integer $d \in \cN_{1,0}$ with $\widetilde{\delta}_{d} \neq 0$. 
\end{conjecture}

Concerning this conjecture, Kurihara proved in \cite{Kur14b} that  the non-degeneracy of the $p$-adic height pairing and the Iwasawa main conjecture for $E/\bQ$ imply Conjecture \ref{conj:1}. 
In the paper \cite{KKS20}, Chan-Ho Kim,  Myoungil Kim, and Hae-Sang Sun called $\widetilde{\delta}_{d}$ Kurihara number at $d$ and  
gave a simple and efficient numerical criterion to verify the Iwasawa main conjecture for $E/\bQ$ by using $\widetilde{\delta}_{d}$, namely, 
they proved in \cite{KKS20} that  Conjecture \ref{conj:1} implies the Iwasawa main conjecture for $E/\bQ$. 
Moreover, Chan-Ho Kim and Nakamura in \cite{KN20} generalized this numerical criterion to the additive reduction case. 
In the present paper, we give the following answer to Conjecture \ref{conj:1}.

\begin{theorem}[Corollary \ref{cor:conj1}]\label{thm:1}
Conjecture \ref{conj:1} is equivalent to the Iwasawa main conjecture for $E/\bQ$. 
\end{theorem}

\begin{remark} 
Skinner and Urban proved in \cite{SU14} that  if there exists a prime $q \neq p$ such that $\mathrm{ord}_{q}(N_{E}) = 1$ and $E[p]$ is ramified at $q$, then the Iwasawa main conjecture for $E$ is valid. 
Here $N_{E}$ is the conductor of $E/\bQ$. 
\end{remark}

Next, let us explain the relation between  the structure of the $p$-Selmer group $\mathrm{Sel}(\bQ, E[p])$ and the analytic quantities $\widetilde{\delta}_{d}$. For that, we use the following terminology of Kurihara in \cite{Kur14b}.

\begin{definition}
We say that an integer $d \in \cN_{1,0}$ is $\delta$-minimal if $\widetilde{\delta}_{d} \neq  0$ and $\widetilde{\delta}_{e} = 0$ for any positive proper divisor $e$ of $d$. 
\end{definition}

Recall that, by the definition of the $p$-Selmer group, the localization map at $\ell$ induces a natural homomorphism 
\[
\mathrm{Sel}(\bQ, E[p]) \longrightarrow E(\bQ_{\ell}) \otimes_{\bZ} \bF_{p}. 
\] 
Let $d \in \cN_{1,0}$ be a $\delta$-minimal integer.  Kurihara proved in \cite{Kur14b} that the natural homomorphism  
\begin{align}\label{map}
\mathrm{Sel}(\bQ, E[p]) \longhookrightarrow \bigoplus_{\ell \mid d}E(\bQ_{\ell}) \otimes_{\bZ} \bF_{p}
\end{align}
is injective (see Remark \ref{rem:inj}), and he conjectured in \cite[Conjecture 2]{Kur14b} that the homomorphism \eqref{map} is an isomorphism. By the definition of  $\cP_{1,0}$, we have 
\[
\dim_{\bF_{p}}(E(\bQ_{\ell}) \otimes_{\bZ} \bF_{p}) = 1
\] 
for each prime divisor $\ell \mid d$, and hence  this conjecture is equivalent to that 
\[
\dim_{\bF_{p}}(\mathrm{Sel}(\bQ, E[p])) = \nu(d), 
\]
where $\nu(d)$ denotes the number of distinct prime divisors of $d$. 
Kurihara showed in \cite[Theorem 4]{Kur14b} that \eqref{map} is an isomorphism in some special cases.  
In the present paper, we solve this conjecture. 




\begin{theorem}[Theorem \ref{thm:mainmain}]\label{thm:main1}
For any $\delta$-minimal integer $d \in \cN_{1,0}$, we have the natural isomorphism 
\[
\mathrm{Sel}(\bQ, E[p]) \stackrel{\sim}{\longrightarrow} \bigoplus_{\ell \mid d}E(\bQ_{\ell}) \otimes_{\bZ} \bF_{p}, 
\]
and hence $\dim_{\bF_{p}}(\mathrm{Sel}(\bQ, E[p])) = \nu(d)$. 
\end{theorem}

\begin{remark}
Theorem \ref{thm:main1} implies that for any integer $d \in \cN_{1,0}$ with $\widetilde{\delta}_{d} \neq 0$, we have 
\[
\dim_{\bF_{p}}(\mathrm{Sel}(\bQ, E[p])) \leq \nu(d). 
\]
Note that the analytic quantity $\widetilde{\delta}_{d}$ is computable, as the author mentioned above. 
\end{remark}

\begin{remark}
After the author had got almost all the results in the present paper, Chan-Ho Kim told the author that he also proved the same result (see \cite{CH}). 
\end{remark}

\begin{remark}\label{rem:counter}
The analogue of Theorem \ref{thm:main1}  for ideal class groups does not hold as Kurihara gave in \cite[\S 5.4]{Kur14b} a counter-example. 
In Remark \ref{rem:self},  we explain what is an important property in order to prove Theorem \ref{thm:main1}. 
\end{remark}

By using the functional equation for modular symbols (see \cite[(1.6.1)]{MT}), 
Kurihara showed in \cite[Lemma 4]{Kur14b} that $w_{E} = (-1)^{\nu(d)}$ for any $\delta$-minimal integer $d \in \cN_{1,0}$. 
Here $w_{E}$ denotes the (global) root number of $E/\bQ$. 
Hence, as an application of Theorem \ref{thm:main1}, we obtain the following result concerning the parity of the order of vanishing of $L$-function $L(E/\bQ,s)$ at $s=1$: 

\begin{corollary}
\label{cor:parity}
Suppose that the Iwasawa main conjecture for $E/\bQ$ holds true. 
Then we have 
\[
\dim_{\bF_{p}}(\mathrm{Sel}(\bQ, E[p])) \equiv \mathrm{ord}_{s=1}(L(E/\bQ,s)) \pmod{2}. 
\]
Moreover, if the $p$-primary part of the Tate--Shafarevich group 
for $E/\bQ$ is finite, then we have 
\[
\mathrm{rank}_{\bZ}(E(\bQ)) \equiv \mathrm{ord}_{s=1}(L(E/\bQ,s)) \pmod{2}. 
\]
\end{corollary}

\begin{proof}
Since we assume that the Iwasawa main conjecture for $E/\bQ$ holds true, Theorem \ref{thm:1} shows that there is a $\delta$-minimal integer $d \in \cN_{1,0}$. 
Then, Theorem \ref{thm:main1}, combined with the fact that $w_{E} = (-1)^{\nu(d)}$, implies that $w_{E} = (-1)^{\dim_{\bF_{p}}(\mathrm{Sel}(\bQ, E[p]))}$. 
Since $w_{E} = (-1)^{\mathrm{ord}_{s=1}(L(E/\bQ,s))}$, we have $\dim_{\bF_{p}}(\mathrm{Sel}(\bQ, E[p])) \equiv \mathrm{ord}_{s=1}(L(E/\bQ,s)) \pmod{2}$. 
\end{proof}

\begin{remark}
Corollary \ref{cor:parity} has already been proved by Nekov\'a\v{r} in \cite{Nek01} (see also \cite{Nek06}), assuming only the condition (a).  
However,  the proof of Corollary \ref{cor:parity} is completely different from that of \cite[Theorem A]{Nek01}. 
\end{remark}

The proof of   Theorem \ref{thm:main1} is based on the theory of  Kolyvagin systems of rank $0$ developed in \cite{sakamoto-koly0}. 
In \S\ref{sec:koly}, we introduce the theory of  Kolyvagin systems. 
In \S\ref{sec:const}, we construct a Kolyvagin system of rank 0 from modular symbols. In \S\ref{sec:main}, we discuss the relation between this Kolyvagin system and the set of the analytic quantities $\{\widetilde{\delta}_{d}\}_{d \in \cN_{1,0}}$, and we  give a proof of Theorem  \ref{thm:main1}. Moreover, by using  the Kolyvagin system constructed in \S\ref{sec:const}, 
we construct an explicit basis of the $p$-Selmer group (see Corollary \ref{cor:basis}). 

In our case, the theory of Kolyvagin systems developed in \cite{sakamoto-koly0} does not work when $p=3$. In Appendix \ref{sec:appendix}, we discuss this problem and extend the theory of Kolyvagin systems so that it can be used even when $p = 3$.

\subsection*{Acknowledgement}

The authors would like to thank Masato Kurihara for his careful reading of the paper and his many helpful suggestions. 
The author  would also like to thank Chan-Ho Kim for helpful comments.

\section{The theory of Kolyvagin system}\label{sec:koly}

In this section, we recall the theory of Kolyvagin systems. 
The contents of this section are based on \cite{MRkoly, sakamoto-koly0}. 

Let $p \geq 3$ be a primes satisfying the hypotheses (a), (b) and (c). 
For notational simplicity, we put 
\[
M/p^{m} := M/p^{m}M
\] 
for any abelian group $M$. 
 Fix integers $n \geq 0$ and $m \geq 1$. 
Let $\bQ_{n}$ denote the $n$-th layer of the cyclotomic $\bZ_{p}$-extension of $\bQ$.
 We then put 
 \[
 R := \bZ_{p}/p^{m}[\Gal(\bQ_{n}/\bQ)] \,\,\, \textrm{ and } \,\,\, T := \mathrm{Ind}_{G_{\bQ_{n}}}^{G_{\bQ}}(E[p^{m}]). 
 \] 
 Note that $T$ satisfies the hypotheses (H.0) -- (H.3) in \cite[\S3.5]{MRkoly}.  However,  $T$ does not satisfy the hypothesis (H.4) in \cite[\S3.5]{MRkoly} when $p=3$. 
 
 \subsection{Selmer structures}
 We introduce two Selmer structures on $T$. 
 Recall that a Selmer structure $\cF$ on $T$ is a collection of the following data:
\begin{itemize}
\item a finite set $S(\cF)$ of rational primes containing $S_{\mathrm{bad}}(E) \cup \{p\}$, 
\item a choice of $R$-submodule $H^{1}_{\cF}(G_{\bQ_{\ell}}, T)$ of $H^{1}(G_{\bQ_{\ell}},T)$ for each prime $\ell \in S(\cF)$. 
\end{itemize}
Here, for any field $K$, we denote by $\overline{K}$ a separable closure of $K$ and set $G_{K} := \Gal(\overline{K}/K)$. 
For each prime $\ell \not\in S(\cF)$, we set 
\[
H^{1}_{\cF}(G_{\bQ_{\ell}}, T) := H^{1}_{\rm ur}(\bQ_{\ell},T) := {\rm ker}\left(H^{1}(\bQ_{\ell}, T) \longrightarrow H^{1}(G_{\bQ_{\ell}^{\rm ur}}, T) \right), 
\] 
where $\bQ_{\ell}^{\rm ur}$ denotes the maximal unramified extension of $\bQ_{\ell}$. 
We define the Selmer module $H^{1}_{\cF}(G_{\bQ}, T)$ by 
\[
H^{1}_{\cF}(G_{\bQ}, T) := \ker\left(H^{1}(G_{\bQ}, T) \longrightarrow \bigoplus_{\ell} H^{1}(G_{\bQ_{\ell}}, T)/H^{1}_{\cF}(G_{\bQ_{\ell}}, T)\right). 
\]
Set $T^{\vee}(1) := \Hom(T, \mu_{p^{\infty}})$. 
For each prime $\ell$, we define 
\[
H^{1}_{\cF^{*}}(G_{\bQ_{\ell}}, T^{\vee}(1)) \subset H^{1}(G_{\bQ_{\ell}}, T^{\vee}(1))
\] 
to be the orthogonal complement of $H^{1}_{\cF}(G_{\bQ_{\ell}}, T)$ with respect to the local Tate pairing. 
Hence we obtain the dual Selmer structure $\cF^{*}$ on $T^{\vee}(1)$. 
Throughout this paper, we regard $\cF^{*}$ as a Selmer structure on $T$ by using the isomorphism $T \cong T^{\vee}(1)$ induced by the Weil pairing.

\begin{theorem}[{\cite[Theorem~2.3.4]{MRkoly}}]\label{pt}
Let $\cF_{1}$ and $\cF_2$ be Selmer structures on $T$ satisfying 
\[
H^{1}_{\cF_{1}}(G_{\bQ_{\ell}}, T) \subset H^{1}_{\cF_{2}}(G_{\bQ_{\ell}}, T)
\]
for all prime $\ell$. Then we have an exact sequence of $R$-modules
\begin{align*}
0 \longrightarrow H^{1}_{\cF_{1}}(G_{\bQ}, T) \longrightarrow H^{1}_{\cF_{2}}(G_{\bQ}, T) 
\longrightarrow \bigoplus_{\ell} H^{1}_{\cF_{2}}(G_{\bQ_{\ell}}, T)/H^{1}_{\cF_{1}}(G_{\bQ_{\ell}}, T) 
\\
\longrightarrow H^{1}_{\cF_{1}^{*}}(G_{\bQ}, T)^\vee \longrightarrow H^{1}_{\cF_{2}^{*}}(G_{\bQ}, T)^\vee \longrightarrow 0, 
\end{align*}
where $\ell$ runs over all the rational primes satisfying 
$H_{\cF_{1}}^{1}(G_{\bQ_{\ell}}, T) \neq H^{1}_{\cF_{2}}(G_{\bQ_{\ell}}, T)$. 
Here $(-)^{\vee} := \Hom(-,\bQ_{p}/\bZ_{p})$. 
\end{theorem}

\begin{lemma}[{\cite[\S3.2]{bss}, \cite[Lemma~3.5.3]{MRkoly}}]\label{lemma:mult}
For any  Selmer structure $\cF$ on $T$,  the canonical map $E[p]  \longhookrightarrow T$ induces an isomorphism
\[
H^{1}_{\cF^*}(G_{\bQ}, E[p]) \stackrel{\sim}{\longrightarrow} H^{1}_{\cF^*}(G_{\bQ}, T)[\fm_{R}]. 
\]
Here $\fm_{R}$ denote the maximal ideal of $R$. In particular, $H^{1}_{\cF^*}(G_{\bQ}, E[p]) = 0$ if and only if $H^{1}_{\cF^*}(G_{\bQ}, T) = 0$. 
\end{lemma}

Following Mazur and Rubin, we define the transversal local condition $H^1_{\rm tr}(G_{\bQ_{\ell}}, T)$ and 
a  Selmer structure $\cF^{a}_{b}(c)$ on $T$. 

\begin{definition}\
\begin{itemize}
\item[(1)] For any integer $d$, we write $\bQ(d)$ for the maximal $p$-subextension of $\bQ(\mu_{d})$. 
\item[(2)] For any prime $\ell$, define  
\begin{align*}
H^1_{\rm tr}(G_{\bQ_{\ell}}, T) &:= \ker \left( H^{1}(G_{\bQ_{\ell}}, T) \longrightarrow H^{1}(G_{\bQ(\ell) \otimes \bQ_{\ell}}, T) \right).   
\end{align*}
We also set $H^{1}_{/*}(G_{\bQ_{\ell}},T) := H^{1}(G_{\bQ_{\ell}},T)/H^{1}_{*}(G_{\bQ_{\ell}},T)$ for $* \in \{{\rm ur}, {\rm tr}\}$. 
\item[(3)] Let $a$, $b$, and $c$ be pairwise relatively prime (square-free) integers. 
Define the Selmer structure $\cF^{a}_{b}(c)$ on $T$ by the following data:
\begin{itemize}
\item $S(\cF^{a}_{b}(c)) := S(\cF) \cup \{ \ell \mid abc \}$,
\item  
$H^{1}_{\cF^{a}_{b}(c)}(G_{\bQ_{\ell}},T) := 
\begin{cases}
H^{1}(G_{\bQ_{\ell}}, T) & \text{if $\ell \mid a$}, 
\\
0 & \text{if $\ell \mid b$}, 
\\
H^1_{\rm tr}(G_{\bQ_{\ell}}, T) & \text{if $\ell \mid c$}, 
\\
H^{1}_{\cF}(G_{\bQ_{\ell}}, T) & \text{otherwise}. 
\end{cases} $
\end{itemize}
Note that $(\cF^{a}_{b}(c))^{*} = (\cF^{*})^{b}_{a}(c)$. 
For simplicity, we will write $\cF^{a}$, $\cF_{b}$, $\cF(c)$, ...  instead of $\cF^{a}_{1}(1)$, $\cF_{b}^{1}(1)$, $\cF^{1}_{1}(c)$, ..., respectively. 
\end{itemize} 
\end{definition}

 \begin{definition}[classical Selmer structure]
We define the classical Selmer structure $\mathrm{\cF}_{\mathrm{cl}}$ on $T$ by the following: 
\begin{itemize}
\item $S(\cF_{\mathrm{cl}}) := S_{\mathrm{bad}}(E) \cup \{p\}$, 
\item $H^{1}_{\cF_{\mathrm{cl}}}(G_{\bQ_{\ell}}, T) := \mathrm{im}\left(\bigoplus_{\fl \mid \ell} E(\bQ_{n, \fl})/p^{m} \longhookrightarrow  H^{1}(G_{\bQ_{\ell}}, T)\right)$ for each prime $\ell \in S(\cF_{\mathrm{cl}})$. 
\end{itemize}
By definition, the Selmer module $H^{1}_{\cF_{\mathrm{cl}}}(G_{\bQ}, T)$ coincides with the classical $p^{m}$-Selmer group $\mathrm{Sel}(\bQ_{n}, E[p^{m}])$ associated with  the elliptic curve $E/\bQ_{n}$. 
We also note that  $\cF_{\mathrm{cl}} = \cF^{*}_{\mathrm{cl}}$. 
 \end{definition}

 \begin{definition}[canonical Selmer structure]
We define the canonical Selmer structure $\cF_{\mathrm{can}}$ on $T$ by 
\[
\cF_{\mathrm{can}} = \cF_{\mathrm{cl}}^{p}. 
\]
 \end{definition}

\begin{lemma}\label{lem:ur=f}
For any prime $\ell \neq p$, we have 
\[
H^{1}_{\cF_{\mathrm{can}}}(G_{\bQ_{\ell}}, T) = H^{1}_{\cF_{\mathrm{cl}}}(G_{\bQ_{\ell}}, T) = H^{1}_{\mathrm{ur}}(G_{\bQ_{\ell}}, T). 
\]
\end{lemma}
\begin{proof}
By definition, it suffices to show that $E(K)/p^{m} = H^{1}_{\mathrm{ur}}(G_{K}, E[p^{m}])$ 
for any unramified $p$-extension $K/\bQ_{\ell}$. Note that $\# H^{1}_{\mathrm{ur}}(G_{K}, E[p^{m}]) = \# H^{0}(G_{K}, E[p^{m}]) 
= \# E(K)/p^{m}$ since $\ell \neq p$. Hence it suffices to show that $E(K)/p^{m} \subset H^{1}_{\mathrm{ur}}(G_{K}, E[p^{m}])$, that is, $E(K) +  p^{m}E(\bQ_{\ell}^{\mathrm{ur}}) = E(\bQ_{\ell}^{\mathrm{ur}})$. This follows from the assumption that $p \nmid \mathrm{Tam}_{\ell}(E)$. 
\end{proof}

\begin{remark}
Let $K/\bQ_{\ell}$ be an unramified extension. 
The assumption that $p \nmid \mathrm{Tam}_{\ell}(E)$ implies that $E(\bQ_{\ell}^{\mathrm{ur}})[p^{\infty}]$ is divisible. 
This fact shows that 
\[
H^{1}_{\mathrm{ur}}(G_{K}, T_{p}(E)) = \ker(H^{1}(G_{K}, T_{p}(E)) \longrightarrow H^{1}(G_{\bQ_{\ell}^{\mathrm{ur}}}, T_{p}(E))\otimes \bQ_{p} )
\] 
and $\mathrm{im}\left(H^{1}_{\mathrm{ur}}(G_{K}, T_{p}(E)) \longrightarrow H^{1}(G_{K}, E[p^{m}])\right) = H^{1}_{\mathrm{ur}}(G_{K}, E[p^{m}])$. 
Therefore,  by Lemma \ref{lem:ur=f}, the canonical Selmer structure in the present paper is the same as the Selmer structure induced by the canonical Selmer structure  defined in \cite[Definition 3.2.1]{MRkoly}. 
\end{remark}

Note that we have the canonical injection $E[p] \longhookrightarrow T$. 

\begin{definition}
We say that a Selmer structure $\cF$ on $T$ is cartesian if the homomorphism 
\[
\mathrm{coker}\left(H^{1}_{\cF}(G_{\bQ_{\ell}}, T) \longrightarrow H^{1}(G_{\bQ_{\ell}}, E[p]) \right) \longrightarrow H^{1}(G_{\bQ_{\ell}}, T)/H^{1}_{\cF}(G_{\bQ_{\ell}}, T)
\]
induced by $E[p] \longhookrightarrow T$  is injective for any prime $\ell \in S(\cF)$. 
\end{definition}

\begin{proposition}\label{prop:sel-car-can}
The Selmer structure $\cF_{\mathrm{can}}$ on $T$ is cartesian. 
\end{proposition}
\begin{proof}
Since we assume $p \nmid \# E(\bF_{p})$, we have $H^{2}(G_{\bQ_{p}}, E[p]) \cong H^{0}(G_{\bQ_{p}}, E[p]) = 0$. 
This fact implies $\mathrm{coker}\left(H^{1}_{\cF_{\mathrm{can}}}(G_{\bQ_{p}}, T) \longrightarrow H^{1}(G_{\bQ_{p}}, E[p]) \right) = 0$. 

Take a prime $\ell \in S_{\mathrm{bad}}(E)$. Since $\bQ_{n}/\bQ$ is unramified at $\ell$, Lemma \ref{lem:ur=f} shows that there are natural injections 
\[
\mathrm{coker}\left(H^{1}_{\cF_{\mathrm{can}}}(G_{\bQ_{\ell}}, T) \longrightarrow H^{1}(G_{\bQ_{\ell}}, E[p]) \right) \longhookrightarrow H^{1}(G_{\bQ_{\ell}^{\mathrm{ur}}}, E[p])
\]
and 
\[
H^{1}(G_{\bQ_{\ell}}, T)/H^{1}_{\cF_{\mathrm{can}}}(G_{\bQ_{\ell}}, T) \longhookrightarrow H^{1}(G_{\bQ_{\ell}^{\mathrm{ur}}}, T) \cong \bigoplus_{\fl \mid \ell} H^{1}(G_{\bQ_{\ell}^{\mathrm{ur}}}, E[p^{m}]). 
\]
Since $p \nmid \mathrm{Tam}_{\ell}(E)$, the module $E(\bQ_{\ell}^{\mathrm{ur}})[p^{\infty}]$ is divisible. 
Hence $E(\bQ_{\ell}^{\mathrm{ur}})[p^{m}] \stackrel{\times p}{\longrightarrow} E(\bQ_{\ell}^{\mathrm{ur}})[p^{m-1}]$ is surjective, and  $H^{1}(G_{\bQ_{\ell}^{\mathrm{ur}}}, E[p]) \longrightarrow H^{1}(G_{\bQ_{\ell}^{\mathrm{ur}}}, E[p^{m}])$ is injective. This completes the proof. 
\end{proof}


\subsection{Structure of local points}

Let $K/\bQ$ be a finite abelian $p$-extension and put 
\[
G := \Gal(K/\bQ). 
\]
Let $\widehat{E}$ denote the formal group associated with $E/\bQ_{p}$ and put 
\[
\widehat{E}(\fm_{K_{p}}) := \bigoplus_{\fp \mid p}\widehat{E}(\fm_{K_{v}}). 
\]
Here $\fm_{L}$ denotes the maximal ideal of the ring of integers of $L$ for any algebraic extension $L/\bQ_{p}$.

\begin{lemma}\label{lemma:coh-vanish} 
We have $\widehat{E}(\fm_{\bQ_{p}})/p = (\widehat{E}(\fm_{K_{p}})/p)^{G}$. 
\end{lemma}
\begin{proof}
Since $p \nmid \#E(\bF_{p})$, Tan proved in \cite[Theorem 2 (a)]{Ki-Seng Tan} that 
\[
H^{1}(G_{\bQ_{p}}, \widehat{E}(\fm_{\overline{\bQ}_{p}})) = 0. 
\] 
Take a prime $\fp \mid p$ of $K$ and put $G_{\fp} := \Gal(K_{\fp}/\bQ_{p})$. 
The injectivity of the inflation map $H^{1}(G_{\fp}, \widehat{E}(\fm_{K_{\fp}})) \longrightarrow H^{1}(G_{\bQ_{p}}, \widehat{E}(\fm_{\overline{\bQ}_{p}}))$  implies  $H^{1}(G_{\fp}, \widehat{E}(\fm_{K_{\fp}})) = 0$. 
Since $K_{\fp}/\bQ_{p}$ is a $p$-extension and $E(\bQ_{p})[p] = 0$, the module $E(K_{\fp})$ is $p$-torsion-free. Hence the vanishing of $H^{1}(G_{\fp}, \widehat{E}(\fm_{K_{\fp}}))$ implies 
\[
\widehat{E}(\fm_{\bQ_{p}})/p = (\widehat{E}(\fm_{K_{\fp}})/p)^{G_{\fp}}. 
\]
Since 
\[
\widehat{E}(\fm_{K_{p}})/p \cong \widehat{E}(\fm_{K_{\fp}})/p \otimes_{\bF_{p}} \bF_{p}[G/G_{\fp}], 
\]
we see that $\widehat{E}(\fm_{\bQ_{p}})/p = (\widehat{E}(\fm_{K_{p}})/p)^{G}$. 
\end{proof}

\begin{proposition}\label{prop:free}
The $\bZ_{p}[G]$-module $\widehat{E}(\fm_{K_{p}})$ is free of rank $1$. 
\end{proposition}
\begin{proof}
By Lemma \ref{lemma:coh-vanish}, we have 
$(\widehat{E}(\fm_{K_{p}})/p)^{G} = \widehat{E}(\fm_{\bQ_{p}})/p \cong \bF_{p}$. 
Since any finitely generated $\bF_{p}[G]$-module is reflexive, we have 
\begin{align*}
((\widehat{E}(\fm_{K_{p}})/p)^{*})_{G} &\cong (((\widehat{E}(\fm_{K_{p}})/p)^{*})_{G})^{**} 
\\
&\cong ((\widehat{E}(\fm_{K_{p}})/p)^{**})^{G})^{*} 
\\
&\cong  (\widehat{E}(\fm_{K_{p}})/p)^{G})^{*} 
\\
&\cong \bF_{p}. 
\end{align*}
Here $(-)^{*} := \Hom_{\bF_{p}[G]}(-, \bF_{p}[G])$. 
Hence $(\widehat{E}(\fm_{K_{p}})/p)^{*}$ is a cyclic $\bF_{p}[G]$-module. 
Furthermore, the fact that $\widehat{E}(\fm_{K_{p}}) \cong \bZ_{p}^{[K \colon \bQ]}$ as $\bZ_{p}$-modules implies that 
\[
(\widehat{E}(\fm_{K_{p}})/p)^{*} \cong \bF_{p}[G]. 
\]
Therefore, $\widehat{E}(\fm_{K_{p}})/p$ is also free of rank $1$, and the $\bZ_{p}[G]$-module $\widehat{E}(\fm_{K_{p}})$ is cyclic. 
Since $\widehat{E}(\fm_{K_{p}}) \cong \bZ_{p}^{[K \colon \bQ]}$, we conclude that 
$\widehat{E}(\fm_{K_{p}}) \cong \bZ_{p}[G]$. 
\end{proof}


\begin{definition}
For any integer $m \geq 1$, we put 
\begin{align*}
H^{1}_{f}(G_{\bQ_{p}}, \mathrm{Ind}_{G_{K}}^{G_{\bQ}}(E[p^{m}])) &:= 
\mathrm{im}\left(\widehat{E}(\fm_{K_{p}})/p^{n} \longhookrightarrow H^{1}(G_{\bQ_{p}}, \mathrm{Ind}_{G_{K}}^{G_{\bQ}}(E[p^{m}]))\right). 
\\
H^{1}_{/f}(G_{\bQ_{p}}, \mathrm{Ind}_{G_{K}}^{G_{\bQ}}(E[p^{m}])) &:= H^{1}(G_{\bQ_{p}}, \mathrm{Ind}_{G_{K}}^{G_{\bQ}}(E[p^{m}]))/H^{1}_{f}(G_{\bQ_{p}}, \mathrm{Ind}_{G_{K}}^{G_{\bQ}}(E[p^{m}])). 
\end{align*}
\end{definition}

\begin{remark}
Since we assme $p \nmid \#E(\bF_{p})$, we have $H^{1}_{f}(G_{\bQ_{p}}, T) = H^{1}_{\cF_{\mathrm{cl}}}(G_{\bQ_{p}}, T)$ when $K = \bQ_{n}$. 
\end{remark}

\begin{corollary}\label{cor:local-str}\ 
\begin{itemize}
\item[(1)] The $\bZ_{p}/p^{m}[G]$-modules 
\[
H^{1}_{f}(G_{\bQ_{p}}, \mathrm{Ind}_{G_{K}}^{G_{\bQ}}(E[p^{m}])) \,\, \textrm{ and } \,\, H^{1}_{/f}(G_{\bQ_{p}}, \mathrm{Ind}_{G_{K}}^{G_{\bQ}}(E[p^{m}]))
\] 
are free of rank $1$. 
\item[(2)] For any subfield $K' \subset K$, we have natural isomorphisms
\begin{align*}
H^{1}_{f}(G_{\bQ_{p}}, \mathrm{Ind}_{G_{K}}^{G_{\bQ}}(E[p^{m}]))_{\Gal(K/K')} &\stackrel{\sim}{\longrightarrow} H^{1}_{f}(G_{\bQ_{p}}, \mathrm{Ind}_{G_{K'}}^{G_{\bQ}}(E[p^{m}])), 
\\
H^{1}_{/f}(G_{\bQ_{p}}, \mathrm{Ind}_{G_{K}}^{G_{\bQ}}(E[p^{m}]))_{\Gal(K/K')} &\stackrel{\sim}{\longrightarrow} H^{1}_{/f}(G_{\bQ_{p}}, \mathrm{Ind}_{G_{K'}}^{G_{\bQ}}(E[p^{m}])). 
\end{align*}
\end{itemize}
\end{corollary}

\begin{proof}
For simplicity, we put $T_{K} := \mathrm{Ind}_{G_{K}}^{G_{\bQ}}(T_{p}(E))$. 
We note that $T_{K}/p^{m} \cong \mathrm{Ind}_{G_{K}}^{G_{\bQ}}(E[p^{m}])$. 
Since $H^{2}(G_{\bQ_{p}}, E[p]) \cong H^{0}(G_{\bQ_{p}}, E[p]) = 0$ and ${\bf R}\Gamma(G_{\bQ_{p}}, T_{K}) \otimes^{\bL}_{\bZ_{p}[G]} \bF_{p} \cong {\bf R}\Gamma(G_{\bQ_{p}}, E[p])$, 
the  perfect complex ${\bf R}\Gamma(G_{\bQ_{p}}, T_{K})$ is of perfect amplitude in $[1,1]$. 
Hence, for any ideal $I$ of $\bZ_{p}[G]$, we have 
\[
H^{1}(G_{\bQ_{p}}, T_{K}) \otimes_{\bZ_{p}[G]} \bZ_{p}[G]/I \stackrel{\sim}{\longrightarrow} H^{1}(G_{\bQ_{p}}, T_{K}/I T_{K}). 
\] 
Furthermore, the local Euler characteristic formula implies that $H^{1}(G_{\bQ_{p}}, T_{K}/I T_{K})$ is a free $\bZ_{p}[G]/I$-module of rank $2$. 

\item[1)] By Proposition \ref{prop:free}, the $\bZ_{p}/p^{m}[G]$-module $H^{1}_{f}(G_{\bQ_{p}}, T_{K}/p^{m})$ is 
free of rank 1. Since $\bZ_{p}/p^{m}[G]$ is a self-injective ring, $H^{1}_{/f}(G_{\bQ_{p}}, T_{K}/p^{m})$ is also free of rank 1.

\item[2)] 
By claim (1), the exact sequence of $\bZ_{p}/p^{m}[G]$-modules 
\[
0 \longrightarrow H^{1}_{f}(G_{\bQ_{p}}, T_{K}/p^{m}) 
\longrightarrow H^{1}(G_{\bQ_{p}}, T_{K}/p^{m}) 
\longrightarrow H^{1}_{/f}(G_{\bQ_{p}}, T_{K}/p^{m}) 
\longrightarrow 0
\]
is split. Hence we obtain the exact sequence of free $\bZ_{p}/p^{m}[\Gal(K'/\bQ)]$-modules 
\begin{align*}
0 \longrightarrow H^{1}_{f}(G_{\bQ_{p}}, T_{K}/p^{m})_{\Gal(K/K')} \longrightarrow &H^{1}(G_{\bQ_{p}}, T_{K}/p^{m})_{\Gal(K/K')} 
\\
&\longrightarrow H^{1}_{/f}(G_{\bQ_{p}}, T_{K}/p^{m})_{\Gal(K/K')}  
\longrightarrow 0. 
\end{align*}
Since $H^{1}(G_{\bQ_{p}}, T_{K}/p^{m})_{\Gal(K/K')}  \stackrel{\sim}{\longrightarrow} H^{1}(G_{\bQ_{p}}, T_{K'}/p^{m})$, 
the homomorphism 
\[
H^{1}_{f}(G_{\bQ_{p}}, T_{K}/p^{m})_{\Gal(K/K')} \longrightarrow H^{1}_{f}(G_{\bQ_{p}}, T_{K'}/p^{m})
\] 
is injective. Hence by claim (1), we obtain isomorphisms 
\begin{align*}
H^{1}_{f}(G_{\bQ_{p}}, T_{K}/p^{m})_{\Gal(K/K')} &\stackrel{\sim}{\longrightarrow} H^{1}_{f}(G_{\bQ_{p}}, T_{K'}/p^{m})
\\
H^{1}_{/f}(G_{\bQ_{p}}, T_{K}/p^{m})_{\Gal(K/K')} &\stackrel{\sim}{\longrightarrow} H^{1}_{/f}(G_{\bQ_{p}}, T_{K'}/p^{m}). 
\end{align*}
\end{proof}

\begin{corollary}\label{cor:sel-car}
The Selmer structure $\cF_{\mathrm{cl}}$ on $T$ is cartesian. 
\end{corollary}
\begin{proof}
By Proposition \ref{prop:sel-car-can}, it suffices to show that the homomorphism 
\[
H^{1}_{/f}(G_{\bQ_{p}}, E[p]) \longrightarrow H^{1}_{/f}(G_{\bQ_{p}}, T)
\] 
is injective. 
Note that this map factors  through $H^{1}_{/f}(G_{\bQ_{p}}, E[p^{m}])$. 
By Corollary \ref{cor:local-str}, the canonical homomorphism $H^{1}_{/f}(G_{\bQ_{p}}, E[p^{m}]) \longrightarrow H^{1}_{/f}(G_{\bQ_{p}}, T)$ is injective. Let us show that $H^{1}_{/f}(G_{\bQ_{p}}, E[p]) \longrightarrow H^{1}_{/f}(G_{\bQ_{p}}, E[p^{m}])$ is injective.  
Since $H^{1}(G_{\bQ_{p}}, E[p^{m}])$ is a free $\bZ_{p}/p^{m}$-module and $H^{1}(G_{\bQ_{p}}, E[p^{m}]) \otimes \bF_{p} \cong H^{1}(G_{\bQ_{p}}, E[p])$,  the canonical homomorphism $H^{1}(G_{\bQ_{p}}, E[p]) \longrightarrow H^{1}(G_{\bQ_{p}}, E[p^{m}])$ is injective. 
By definition, we have 
\[
H^{1}_{f}(G_{\bQ_{p}}, E[p^{m}]) \otimes \bF_{p} = \widehat{E}(\fm_{\bQ_{p}})/p^{m} \otimes \bF_{p} = \widehat{E}(\fm_{\bQ_{p}})/p = H^{1}_{f}(G_{\bQ_{p}}, E[p]). 
\]
Since $H^{1}_{f}(G_{\bQ_{p}}, E[p^{m}]) \cong \bZ_{p}/p^{m}$ by Corollary \ref{cor:local-str}, we see that  the canonical homomorphism $H^{1}_{/f}(G_{\bQ_{p}}, E[p]) \longrightarrow H^{1}_{/f}(G_{\bQ_{p}}, E[p^{m}])$ is injective. 
\end{proof}


\subsection{Kolyvagin systems of rank $1$}
In this subsection, we recall  the definition of Kolyvagin systems of rank $1$ introduced by Mazur and Rubin in \cite{MRkoly}. We set 
\[
\cP_{m,n} := \{\ell \not\in S_{\mathrm{bad}}(E) \mid E(\bF_{\ell})[p^{m}] \cong \bZ/p^{m} \, \textrm{ and } \,  \ell \equiv 1 \pmod{p^{\max\{m,n+1\}}} \}. 
\]
For any prime $\ell \in \cP_{m,n}$, the $R$-module $H^{1}_{\mathrm{ur}}(G_{\bQ_{\ell}}, T) \cong T/(\mathrm{Fr}_{\ell}-1)T$ is free of rank $1$. 
Moreover,  by \cite[Lemmas~1.2.1, 1.2.3 and 1.2.4]{MRkoly}, we have 
\[
H^{1}(G_{\bQ_{\ell}},T) = H^{1}_{\rm ur}(G_{\bQ_{\ell}},T) \oplus H^{1}_{\rm tr}(G_{\bQ_{\ell}},T)
\]
and the $R$-modules $H^{1}_{\rm tr}(G_{\bQ_{\ell}},T)$, $H^{1}_{/{\rm ur}}(G_{\bQ_{\ell}},T)$, and $H^{1}_{/{\rm tr}}(G_{\bQ_{\ell}},T)$ are free of rank $1$. 
Let $\cN_{m,n}$ denote the set of square-free products in $\cP_{m,n}$.  
For each integer $d \in \cN_{m,n}$, we  put 
\[
G_{d} := \bigotimes_{\ell \mid d}\Gal(\bQ(\ell)/\bQ). 
\]
For any prime $\ell \in \cP_{m,n}$, we have two homomorphisms  
\begin{align*}
v_\ell &\colon H^1(G_{\bQ}, T) \xrightarrow{{\rm loc}_\ell} H^1(G_{\bQ_{\ell}}, T) \longrightarrow H^1_{/{\rm ur}}(G_{\bQ_{\ell}}, T), 
\\
\varphi_\ell^{\rm fs} &\colon H^1(G_{\bQ}, T) \xrightarrow{{\rm loc}_\ell}  H^1(G_{\bQ_{\ell}}, T) \xrightarrow{{\rm pr}_{\rm ur}} H^1_{\rm ur}(G_{\bQ_{\ell}}, T) 
\xrightarrow{\phi_\ell^{\rm fs}} H^1_{/{\rm ur}}(G_{\bQ_{\ell}}, T) \otimes_{\bZ} G_{\ell}. 
\end{align*}
Here $\phi^{\mathrm{fs}}_{\ell}$ is the finite-singular comparison map defined in \cite[Definition 1.2.2]{MRkoly} and 
${\rm pr}_{\rm ur}$ denotes the projection map with respect to the decomposition 
$ H^1(G_{\bQ_{\ell}}, T) = H^1_{\rm  ur}(G_{\bQ_{\ell}}, T) \oplus H^1_{\rm tr}(G_{\bQ_{\ell}}, T)$.

\begin{definition}
We define the module ${\rm KS}_1(T, \cF_{\mathrm{can}})$ of Kolyvagin systems of rank $1$ to be the set of elements 
\[
(\kappa_d)_{d \in \cN_{m,n}} \in \prod_{d \in \cN_{m,n}} H^1_{\cF_{\mathrm{can}}(d)}(G_{\bQ}, T) \otimes_{\bZ} G_{n}
\]
satisfying the finite-singular relation 
\[
v_\ell(\kappa_d)=\varphi_\ell^{\rm fs}(\kappa_{d/\ell})
\]
for any integer $d \in \cN_{m,n}$ and any prime  $\ell \mid d$. 
\end{definition}

For any integer $d$, we denote by  $\nu(d) \in \bZ_{\geq 0}$ the number of prime divisors of $d$. 

\begin{lemma}\label{lem:free-can}
Let $a,b,c \in \cN_{m,n}$ be pairwise relatively prime integers with $\nu(a) - \nu(b) \geq 1$. 
If $H^{1}_{(\cF_{\mathrm{can}}^{*})^{b}_{a}(c)}(G_{\bQ}, E[p]) = 0$, then the $R$-module  $H^{1}_{(\cF_{\mathrm{can}})^{a}_{b}(c)}(G_{\bQ}, T)$ is free of rank $\nu(a)-\nu(b)+1$. 
\end{lemma}

\begin{proof}
Since $\cF_{\mathrm{can}}$ is cartesian by Proposition \ref{prop:sel-car-can},  so is $(\cF_{\mathrm{can}})^{a}_{b}(c)$ by \cite[Corollary 3.18]{sakamoto}. 
By \cite[Proposition 6.2.2]{MRkoly}, we have 
\[
\chi(\cF_{\mathrm{can}}) := \dim_{\bF_{p}}(H^{1}_{\cF_{\mathrm{can}}}(G_{\bQ}, E[p])) - \dim_{\bF_{p}}(H^{1}_{\cF_{\mathrm{can}}^{*}}(G_{\bQ}, E[p])) = 1, 
\]
and \cite[Corollary 3.21]{sakamoto} implies $\chi((\cF_{\mathrm{can}})^{a}_{b}(c)) = \nu(a)-\nu(b)+1$. 
Hence this lemma follows from \cite[Lemma 4.6]{sakamoto}. 
\end{proof}



\subsection{Kolyvagin systems of rank $0$}
In this subsection, we recall the definition of Kolyvagin system of rank $0$ in our previous paper \cite{sakamoto-koly0}. 
Fix an isomorphism 
\[
H^{1}_{/{\rm ur}}(G_{\bQ_{\ell}},T) \cong R
\] 
for each prime $\ell \in \cP_{m,n}$.  
We then have homomorphisms
\begin{align*}
v_\ell &\colon H^1(G_{\bQ_{\ell}}, T) \longrightarrow H^{1}_{/{\rm ur}}(G_{\bQ_{\ell}},T) \cong R, 
\\
\varphi_\ell^{\rm fs} &\colon H^1(G_{\bQ_{\ell}}, T) \longrightarrow  H^1_{/{\rm ur}}(G_{\bQ_{\ell}}, T)\otimes_{\bZ} G_\ell \cong R \otimes_{\bZ} G_\ell.  
\end{align*}
We put $\cM_{m,n} := \{(d, \ell) \in \cN_{m,n} \times \cP_{m,n} \mid \text{$\ell$ is coprime to $d$}\}$. 

\begin{definition}\label{def:koly0}
A Kolyvagin system of rank $0$ is an element
\[
(\kappa_{d,\ell})_{(d,\ell) \in \cM_{m,n}} \in \prod_{(d,\ell) \in \cM_{m,n}} H^1_{\cF_{\mathrm{cl}}^{\ell}(d)}(G_{\bQ},T) \otimes_{\bZ} G_n
\]
which satisfies the following relations for any elements $(d,\ell), (d,q), (d\ell, q) \in \cM_{m,n}$: 
\begin{align*}
v_\ell(\kappa_{d\ell, q}) &= \varphi_\ell^{\rm fs}(\kappa_{d,q}), 
\\
v_{\ell}(\kappa_{1,\ell}) &= v_{q}(\kappa_{1,q}), 
\\
v_{q}(\kappa_{d\ell, q}) &= - \varphi_{\ell}^{\rm fs}(\kappa_{d,\ell}). 
\end{align*}
We denote by ${\rm KS}_0(T, \cF_{\mathrm{cl}})$ the module of  Kolyvagin systems of rank $0$. 
For any Kolyvagin system $\kappa \in {\rm KS}_{0}(T,\cF_{\mathrm{cl}})$ and any element $(d, \ell) \in \cM_{\mathrm{ur}}$, we put 
\[
\delta(\kappa)_{d} := v_{\ell}(\kappa_{d,\ell}) \in R \otimes_{\bZ} G_{d}. 
\]
Note that, by the definition of  Kolyvagin system of rank $0$, the element $\delta(\kappa)_{d}$ is independent of the choice of the prime $\ell \nmid d$. 
Hence we obtain a homomorphism 
\[
\delta \colon {\rm KS}_{0}(T,\cF_{\mathrm{cl}}) \longrightarrow \prod_{d \in \cN_{m,n}}R \otimes_{\bZ} G_{d}. 
\]
\end{definition}

Note that $\cF_{\mathrm{cl}} = \cF_{\mathrm{cl}}^{*}$.

\begin{lemma}\label{lem:free-ur}
Let $a,b,c \in \cN_{m,n}$ be pairwise relatively prime integers with $\nu(a) \geq \nu(b)$. 
If $H^{1}_{(\cF_{\mathrm{cl}})^{b}_{a}(c)}(G_{\bQ}, E[p]) = 0$, then the $R$-module  $H^{1}_{(\cF_{\mathrm{cl}})^{a}_{b}(c)}(G_{\bQ}, T)$ is free of rank $\nu(a)-\nu(b)$. 
\end{lemma}
\begin{proof}
Since $H^{1}_{(\cF_{\mathrm{cl}})^{b}_{a}(c)}(G_{\bQ}, E[p]) = 0$, Lemma \ref{lemma:mult} shows that 
$H^{1}_{(\cF_{\mathrm{cl}})^{b}_{a}(c)}(G_{\bQ}, T) = 0$. 
Hence applying Theorem \ref{pt} with $\cF_{1} = (\cF_{\mathrm{cl}})^{a}_{b}(c)$ and $\cF_{2} = (\cF_{\mathrm{can}})^{a}_{b}(c)$, we obtain an exact sequence 
\[
0 \longrightarrow H^{1}_{(\cF_{\mathrm{cl}})^{a}_{b}(c)}(G_{\bQ}, T) \longrightarrow H^{1}_{(\cF_{\mathrm{can}})^{a}_{b}(c)}(G_{\bQ}, T) \longrightarrow H^{1}_{/f}(G_{\bQ_{p}}, T) \longrightarrow0. 
\]
Hence this lemma follows from Corollary \ref{cor:local-str} and Lemma \ref{lem:free-can}. 
\end{proof}

When $p>3$, the following theorem is proved in {\cite[Proposition 5.6, Theorem 5.8]{sakamoto-koly0}}. 
When $p=3$, it is proved in Appendix \ref{sec:appendix}.

\begin{theorem}\label{thm:koly0}\ 
\begin{itemize}
\item[(1)] For any element $(d, \ell) \in \cM_{m,n}$ satisfying $H^{1}_{(\cF_{\mathrm{cl}})_{\ell}(d)}(G_{\bQ}, E[p])=0$, the projection map 
\[
{\rm KS}_{0}(T,\cF_{\mathrm{cl}}) \longrightarrow  H^{1}_{\cF_{\mathrm{cl}}^{\ell}(d)}(G_{\bQ},T) \otimes_{\bZ} G_{d}
\]
is an isomorphism. 
In particular, the $R$-module ${\rm KS}_{0}(T,\cF_{\mathrm{cl}})$ is free of rank $1$. 
\item[(2)] For any basis $\kappa \in {\rm KS}_{0}(T,\cF_{\mathrm{cl}})$ and any integer $d \in \cN_{m,n}$, we have 
\[
R \cdot \delta(\kappa)_{d} = \mathrm{Fitt}_{R}^{0}(H^{1}_{\cF_{\mathrm{cl}}(d)}(G_{\bQ}, T)^{\vee}). 
\]
\end{itemize}
\end{theorem}

\begin{remark}
For any Selmer structure $\cF$ on $E[p]$ with $\chi(\cF) \geq 0$, there are infinitely many integers $d \in \cN_{m,n}$ satisfying $H^{1}_{\cF^{*}(d)}(G_{\bQ}, E[p]) = 0$ (see \cite[Corollary 4.1.9]{MRkoly}). 
\end{remark}

\begin{corollary}\label{cor:inj}
The homomorphism $\delta$ is injective. 
\end{corollary}
\begin{proof}
Take an integer $d \in \cN_{m,n}$ with $H^{1}_{\cF_{\mathrm{cl}}(d)}(G_{\bQ}, E[p]) = 0$. Then by Theorem \ref{thm:koly0}, we have $\delta(\kappa)_{d} \in R^{\times}$. Since the $R$-module ${\rm KS}_{0}(T,\cF_{\mathrm{cl}})$ is free of rank $1$ by Theorem \ref{thm:koly0},  the map $\delta$ is injective. 
\end{proof}

\subsection{Map from Kolyvagin systems of rank $1$ to Kolyvagin systems of rank $0$}\label{sec:rank reduction}
Fix an isomorphism 
\[
H^{1}_{/f}(G_{\bQ_{p}}, T) \cong R. 
\]
Then we obtain a homomorphism $\varphi \colon H^{1}(G_{\bQ}, T) \longrightarrow H^{1}_{/f}(G_{\bQ_{p}}, T) \cong R$. 
We also denote by $\varphi \colon \mathrm{KS}_{1}(T, \cF_{\mathrm{can}}) \longrightarrow \prod_{d \in \cN_{m,n}}R \otimes_{\bZ} G_{d}$ 
the homomorphism induced by $\varphi$. 
In this subsection, we construct a natural map $\mathrm{KS}_{1}(T, \cF_{\mathrm{can}})  \longrightarrow \mathrm{KS}_{0}(T, \cF_{\mathrm{cl}})$ such that the diagram 
\begin{align}
\begin{split}\label{diag1}
\xymatrix{
\mathrm{KS}_{1}(T, \cF_{\mathrm{can}}) \ar[r] \ar[rd]^-{\varphi} & \mathrm{KS}_{0}(T, \cF_{\mathrm{cl}}) \ar[d]^-{\delta}
\\
& \prod_{d \in \cN_{m,n}}R \otimes_{\bZ}G_{d}
}
\end{split}
\end{align}
commutes. 
In order to construct this map, we introduce  the module of  Stark systems.


For any $R$-module $M$, we put 
\[
M^{*} := \Hom_{R}(M, R) \,\, \textrm{ and } \,\,  {\bigcap}^{r}_{R}M := \left({\bigwedge}^{r}_{R}M^{*}\right)^{*}
\] 
for any integer $r \geq 0$. Since the functor $M \mapsto M^{*}$ is exact, an $R$-homomorphism $\phi \colon M \longrightarrow F$, where $F$ is free of rank $1$, induces a natural homomorphism 
\[
\phi \colon {\bigcap}^{r+1}_{R}M \longrightarrow F \otimes_{R} {\bigcap}^{r}_{R}\ker(\phi). 
\]

\begin{definition}\label{def:stark}
Let $\cF$ be a Selmer structure on $T$. 
For any integers $d \in \cN_{m,n}$ and  $r \geq 0$, define 
\begin{align*} 
W_{d} &:=  \bigoplus_{\ell \mid d} H^{1}_{/{\rm ur}}(G_{\bQ_{\ell}}, T)^*, 
\\
X_{d}^{r}(T, \cF) &:=  {\bigcap}^{r + \nu(d)}_R H^{1}_{\cF^{d}}(G_{\bQ}, T) \otimes_{R} {\rm det}(W_{d}). 
\end{align*}
Then for any positive  divisor $e$ of $d$, the exact sequence 
\[
0 \longrightarrow H^{1}_{\cF^{e}}(G_{\bQ}, T) 
\longrightarrow H^{1}_{\cF^{d}}(G_{\bQ}, T) 
\longrightarrow \bigoplus_{\ell \mid \frac{d}{e}}H^{1}_{/{\rm ur}}(G_{\bQ_{\ell}}, T)
\]  
induces a natural homomorphism 
\[
\Phi_{d,e} \colon X_{d}^{r}(T, \cF) \longrightarrow X_{e}^{r}(T, \cF)
\]
(see \cite[Definition 2.3]{sakamoto}). 
If $f \mid e \mid d$, then we have $\Phi_{d,f} = \Phi_{e,f} \circ \Phi_{d,e}$ (see \cite[Proposition 2.4]{sakamoto}),  and we obtain the module of Stark systems of rank $r$ 
\[
{\rm SS}_{r}(T, \cF) := \varprojlim_{d \in \cN_{m,n}} X_{d}^{r}(T, \cF).  
\]
\end{definition}

Since we have the isomorphisms  
\[
H^{1}_{\rm ur}(G_{\bQ_{\ell}},T) \stackrel{\phi_{\ell}^{\mathrm{fs}}}{\longrightarrow}   H^1_{/{\rm ur}}(G_{\bQ_{\ell}}, T) \otimes_{\bZ} G_{\ell} 
\,\,\, \textrm{ and } \,\,\, 
H^{1}_{\rm ur}(G_{\bQ_{\ell}},T) \stackrel{\sim}{\longrightarrow}  H^{1}_{/{\rm tr}}(G_{\bQ_{\ell}}, T) 
\]
for any prime $\ell \mid d$, we see that the exact sequence 
\[
0 \longrightarrow H^{1}_{\cF_{\mathrm{can}}(d)}(G_{\bQ}, T) 
\longrightarrow H^{1}_{\cF_{\mathrm{can}}^{d}}(G_{\bQ}, T) 
\longrightarrow \bigoplus_{\ell \mid d}H^{1}_{/{\rm tr}}(G_{\bQ_{\ell}}, T)
\]  
induces a natural homomorphism 
\[
\Pi_{d} \colon X_{d}^{1}(T, \cF_{\mathrm{can}}) \longrightarrow {\bigcap}^{1}_{R}H^{1}_{\cF_{\mathrm{can}}(d)}(G_{\bQ}, T) \otimes_{\bZ} G_{d} = 
H^{1}_{\cF_{\mathrm{can}}(d)}(G_{\bQ}, T) \otimes_{\bZ} G_{d}, 
\]
and we obtain 
\[
\mathrm{Reg}_{1} \colon {\rm SS}_{1}(T, \cF_{\mathrm{can}}) \longrightarrow {\rm KS}_{1}(T, \cF_{\mathrm{can}}); \, (\epsilon_{d})_{d \in \cN_{m,n}} \mapsto ((-1)^{\nu(d)}\Pi_{d}(\epsilon_{d}))_{d \in \cN_{m,n}} 
\]  
(see \cite[Proposition 4.3]{sbA} or \cite[Proposition 12.3]{MRselmer}). 
The following important proposition is proved by Mazur and Rubin in \cite[Proposition 12.4]{MRselmer} when $p>3$ (see also {\cite[Theorem 5.2(i)]{bss} and \cite[Theorem 3.17]{sakamoto-bessatsu}). 
When $p=3$, this proposition is proved in Appendix \ref{sec:appendix}. 

\begin{proposition}\label{prop:reg-isom}
The map 
\[
\mathrm{Reg}_{1} \colon {\rm SS}_{1}(T, \cF_{\mathrm{can}}) \longrightarrow {\rm KS}_{1}(T, \cF_{\mathrm{can}})
\] 
is an isomorphism. 
\end{proposition}


For any integer $d \in \cN_{m,n}$, the exact sequence 
\[
0 \longrightarrow H^{1}_{\cF_{\mathrm{cl}}(d)}(G_{\bQ}, T) 
\longrightarrow H^{1}_{\cF_{\mathrm{cl}}^{d}}(G_{\bQ}, T) 
\longrightarrow \bigoplus_{\ell \mid d}H^{1}_{/{\rm tr}}(G_{\bQ_{\ell}}, T)
\]  
induces a natural homomorphism 
\[
\Pi_{d}' \colon X_{d}^{0}(T, \cF_{\mathrm{cl}}) \longrightarrow {\bigcap}^{0}_{R}H^{1}_{\cF_{\mathrm{can}}(d)}(G_{\bQ}, T) \otimes_{\bZ} G_{d} = 
R \otimes_{\bZ} G_{d}.
\]
Hence we obtain a homomorphism  
\[
\psi \colon \mathrm{SS}_{0}(T, \cF_{\mathrm{cl}}) \longrightarrow \prod_{d \in \cN_{m,n}}R \otimes_{\bZ} G_{d}; (\epsilon_{d})_{d\in\cN_{m,n}} \mapsto (\Pi_{d}'(\epsilon_{d}))_{d\in\cN_{m,n}}. 
\]
In \cite[\S 5.2]{sakamoto-koly0}, we construct the canonical homomorphism 
\[
\mathrm{Reg}_{0} \colon \mathrm{SS}_{0}(T, \cF_{\mathrm{cl}} ) \longrightarrow \mathrm{KS}_{0}(T, \cF_{\mathrm{cl}})
\]
such that the diagram 
\begin{align}
\begin{split}\label{diag2}
\xymatrix{
 \mathrm{SS}_{0}(T, \cF_{\mathrm{cl}} ) \ar[r]^-{\mathrm{Reg}_{0}} \ar[rd]^-{\psi}  & \mathrm{KS}_{0}(T, \cF_{\mathrm{cl}}) \ar[d]^-{\delta}
\\
& \prod_{d \in \cN_{m,n}}R \otimes_{\bZ} G_{d}
}
\end{split}
\end{align}
commutes. 

For any integer $d \in \cN_{m,n}$, we have an exact sequece
\[
0 \longrightarrow H^{1}_{\cF^{d}_{\mathrm{cl}}}(G_{\bQ}, T) \longrightarrow H^{1}_{\cF^{d}_{\mathrm{can}}}(G_{\bQ}, T) 
\stackrel{\varphi}{\longrightarrow} R. 
\]
This exact sequence induces a homomorphism $X_{d}^{1}(T, \cF_{\mathrm{can}}) \longrightarrow X_{d}^{0}(T, \cF_{\mathrm{cl}})$, and we obtain a homomorphism $\mathrm{SS}_{1}(T, \cF_{\mathrm{can}} ) \longrightarrow \mathrm{SS}_{0}(T, \cF_{\mathrm{cl}} )$. 
By construction, the diagram 
\begin{align}
\begin{split}\label{diag3}
\xymatrix{
 \mathrm{SS}_{1}(T, \cF_{\mathrm{can}} ) \ar[r] \ar[d]^-{\mathrm{Reg}_{1}}  &   \mathrm{SS}_{0}(T, \cF_{\mathrm{cl}} ) \ar[d]^-{\psi}
 \\
 \mathrm{KS}_{1}(T, \cF_{\mathrm{can}}) \ar[r]^-{\varphi} 
& \prod_{d \in \cN_{m,n}}R \otimes_{\bZ} G_{d}
}
\end{split}
\end{align}
commutes. Since $\mathrm{Reg}_{1}$ is an isomorphism, by using the commutative diagrams \eqref{diag2} and \eqref{diag3}, we obtain  the homomorphism $\mathrm{KS}_{1}(T, \cF_{\mathrm{can}})  \longrightarrow \mathrm{KS}_{0}(T, \cF_{\mathrm{cl}})$ such that the diagram \eqref{diag1} commutes.

\section{Construction of the Kolyvagin system of rank $0$ from modular symbols}\label{sec:const}

Let $p \geq 3$ be a prime satisfying the hypotheses (a), (b), and (c). 
For any finite abelian extension $K/\bQ$, we put 
\[
R_{K} := \bZ_{p}[\Gal(K/\bQ)] \,\,\, \textrm{ and } \,\,\, T_{K} :=  \mathrm{Ind}_{G_{K}}^{G_{\bQ}}(T_{p}(E)). 
\]



\subsection{Modular sysmbols}\label{sec:modular}

We recall the definition of the Mazur--Tate elements. 
For any integer $d \geq 1$, we define  the modular element $\widetilde{\theta}_{\bQ(\mu_{d})}$ by 
\[
\widetilde{\theta}_{\bQ(\mu_{d})} := \sum^{d}_{\substack{a=1 \\ (a,d)=1}}\frac{\mathrm{Re}([a/d])}{\Omega_{E}^{+}}\sigma_{a} \in \bQ[\Gal(\bQ(\mu_{d})/\bQ)].  
\]
Here $\sigma_{a} \in \Gal(\bQ(\mu_{d})/\bQ)$ is the element satisfying $\sigma_{a}(\zeta) = \zeta^{a}$ for any $\zeta \in \mu_{d}$. 
For any integer $e \mid d$, we put 
\[
\nu_{d,e} \colon R_{\bQ(\mu_{e})} \longrightarrow R_{\bQ(\mu_{d})}; \, x \mapsto \sum_{\sigma \in \Gal(\bQ(\mu_{d})/\bQ(\mu_{e}))}\sigma x. 
\]
Define $\cP := \{\ell \neq p\mid \textrm{$E$ has good reduction at $\ell$\,}\}$ and $\cN$ denotes the set of square-free products in $\cP$. 
Since $G_{\bQ} \longrightarrow  \mathrm{GL}(E[p])$ is surjective, for any integers $d \in \cN$ and $n \geq 1$, we have 
\[
\widetilde{\theta}_{\bQ(\mu_{dp^{n}})} \in R_{\bQ(\mu_{dp^{n}})} 
\]
(see \cite{SG1989}). 
Let $\alpha \in \bZ_{p}^{\times}$ be the unit root of $x^{2}-a_{p}x+p=0$. 
We set 
\[
\vartheta_{\bQ(\mu_{dp^{n}})} := \alpha^{-n}(\widetilde{\theta}_{\bQ(\mu_{dp^{n}})}-\alpha^{-1}\nu_{dp^{n}, dp^{n-1}}(\widetilde{\theta}_{\bQ(\mu_{dp^{n-1}})})) \in R_{\bQ(\mu_{dp^{n}})}. 
\]
Then the set $\{\vartheta_{\bQ(\mu_{dp^{n}})}\}_{n \geq 1}$ is a projective system and we get an element 
\[
\vartheta_{\bQ(\mu_{dp^{\infty}})} := \varprojlim_{n}\vartheta_{\bQ(\mu_{dp^{n}})} \in \varprojlim_{n}R_{\bQ(\mu_{dp^{n}})} 
=: \Lambda_{\bQ(\mu_{dp^{\infty}})}. 
\]

\begin{remark}\label{rem:bottom}
Note that for any positive integer $d \nmid p$, we have 
\[
\vartheta_{\bQ(\mu_{d})} = \left(1 - \alpha^{-1}\sigma_{p}\right)\left(1 - \alpha^{-1}\sigma^{-1}_{p}\right) \widetilde{\theta}_{\bQ(\mu_{d})}. 
\]
The assumption (c) shows that  $\alpha \not\equiv 1 \pmod{p}$, and   $\left(1 - \alpha^{-1}\sigma_{p}\right)\left(1 - \alpha^{-1}\sigma^{-1}_{p}\right)$ is a unit in $R_{\bQ(\mu_{d})}$. 
\end{remark}

For any prime $\ell$ with $\ell \nmid d$, let $\pi_{\ell d,d} \colon \Lambda_{\bQ(\mu_{\ell dp^{\infty}})} \longrightarrow \Lambda_{\bQ(\mu_{ dp^{\infty}})}$ denote the natural projection map, and we have 
\[
\pi_{\ell d, d}(\vartheta_{\bQ(\mu_{\ell dp^{\infty}})}) = (a_{\ell}-\sigma_{\ell}-\sigma_{\ell}^{-1})\vartheta_{\bQ(\mu_{dp^{\infty}})}. 
\] 
Here $a_{\ell} := \ell + 1 - \#E(\bF_{\ell})$. 
Following Kurihara in \cite[page 324]{Kur14b}, for any positive divisor $e$ of $d$, we put 
\begin{align*}
\alpha_{d,e} &:= \left(\prod_{\ell \mid \frac{d}{e}}(-\sigma_{\ell}^{-1})\right)\vartheta_{\bQ(\mu_{ep^{\infty}})} \in \Lambda_{\bQ(\mu_{ep^{\infty}})}, 
\\
\xi_{\bQ(\mu_{dp^{\infty}})} &:= \sum_{e \mid d}\nu_{d,e}(\alpha_{d,e}) \in \Lambda_{\bQ(\mu_{dp^{\infty}})}. 
\end{align*}
Here $e$ runs over the set of positive divisors of $d$. We also put 
\begin{align*}
\widetilde{\xi}_{\bQ(\mu_{dp^{\infty}})} := \left(\prod_{\ell \mid d}(-\ell \sigma_{\ell})^{-1}\right)\xi_{\bQ(\mu_{dp^{\infty}})}. 
\end{align*}

\begin{definition}
For any prime $\ell \in \cP$, we define the Frobenius polynomial at $\ell$ by  
\[
P_{\ell}(t) := \det(1- t \sigma^{-1}_{\ell}\mid T) = t^{2} - \ell^{-1}a_{\ell}t + \ell^{-1}. 
\]
\end{definition}

\begin{proposition}\label{prop:rel}
For any integer $d \in \cN$ and any prime $\ell \in \cP$ with $\ell \nmid d$,  we have 
\[
\pi_{d\ell, d}(\widetilde{\xi}_{\bQ(\mu_{\ell dp^{\infty}})}) = P_{\ell}(\sigma^{-1}_{\ell})\widetilde{\xi}_{\bQ(\mu_{dp^{\infty}})}. 
\]
\end{proposition}
\begin{proof}
Kurihara showed in \cite[page 325, (7)]{Kur14b} that 
\begin{align*}
\pi_{d\ell, d}(\xi_{\bQ(\mu_{\ell dp^{\infty}})}) 
&= (-\sigma_{\ell}+a_{\ell}-\ell \sigma_{\ell}^{-1})\xi_{\bQ(\mu_{dp^{\infty}})}
\\
&= (-\ell \sigma_{\ell}) P_{\ell}(\sigma^{-1}_{\ell})\xi_{\bQ(\mu_{dp^{\infty}})}, 
\end{align*}
which implies $\pi_{d\ell, d}(\widetilde{\xi}_{\bQ(\mu_{\ell dp^{\infty}})}) = P_{\ell}(\sigma^{-1}_{\ell})\widetilde{\xi}_{\bQ(\mu_{dp^{\infty}})}$. 
\end{proof}


\subsection{Coleman maps}

Let $K/\bQ$ be a $p$-abelian extension at which $p$ is unramified, and we denote by $K_{\infty}/K$ the cyclotomic $\bZ_{p}$-extension. 
Put 
\[
\Lambda_{K_{\infty}} := \bZ_{p}[[\Gal(K_{\infty}/\bQ)]] \,\,\, \textrm{ and } \,\,\, \bT_{K_{\infty}} := \varprojlim_{n} T_{K_{n}}, 
\]
where $K_{n}$ denotes the $n$-th layer of the cyclotomic $\bZ_{p}$-extension $K_{\infty}/K$. 
We note that the $\Lambda_{K_{\infty}}$-module $H^{1}_{/f}(G_{\bQ}, \bT_{K_{\infty}})$ is free of rank $1$ by Corollary \ref{cor:local-str}. 

The following theorem follows from the works of Perrin-Riou in \cite{PR94} and Kato in \cite{Kato04} 

\begin{theorem}[{\cite[Theorem 16.4, Theorem 16.6, and Proposition 17.11]{Kato04}}]\label{thm:zeta}
There exists an isomorphism  
\[
\fL_{K_{\infty}} \colon H^{1}_{/f}(G_{\bQ}, \bT_{K_{\infty}}) \otimes_{\bZ_{p}} \bQ_{p} \stackrel{\sim}{\longrightarrow} \Lambda_{K_{\infty}} \otimes_{\bZ_{p}} \bQ_{p}  
\]
such that 
\begin{itemize}
\item[(i)] tha diagram 
\[
\xymatrix{
 H^{1}_{/f}(G_{\bQ}, \bT_{K_{\infty}}) \otimes_{\bZ_{p}} \bQ_{p} \ar[d] \ar[rr]^-{\fL_{K_{\infty}}} & & \Lambda_{K_{\infty}} \otimes_{\bZ_{p}} \bQ_{p}  \ar[d]
\\
 H^{1}_{/f}(G_{\bQ}, \bT_{L_{\infty}}) \otimes_{\bZ_{p}} \bQ_{p} \ar[rr]^-{\fL_{L_{\infty}} } & & \Lambda_{L_{\infty}} \otimes_{\bZ_{p}} \bQ_{p}  
}
\]
commutes for any field $L \subset K$, where the vertical maps are the natural projections, 
\item[(ii)] $\fL_{\bQ_{\infty}}( H^{1}_{/f}(G_{\bQ}, \bT_{\bQ_{\infty}})) =  \Lambda_{\bQ_{\infty}}$, 
\item[(iii)] there is an element $z_{K_{\infty}} \in H^{1}(G_{\bQ}, \bT_{K_{\infty}})$ such that $\fL_{K_{\infty}}(\mathrm{loc}_{p}^{/f}(z_{K_{\infty}})) = \widetilde{\xi}_{K_{\infty}}$, where  $\mathrm{loc}_{p}^{/f} \colon H^{1}(G_{\bQ}, \bT_{K_{\infty}}) \longrightarrow H^{1}_{/f}(G_{\bQ_{p}}, \bT_{K_{\infty}})$ denotes the localization  homomorphism. 
\end{itemize}
\end{theorem}

\begin{remark}
Note that the integrality of the element $z_{K_{\infty}}$ follows from the assumption (b) (see \cite[Theorem 6.1]{Kataoka}).  
\end{remark}


%


\subsection{Euler systems}
In this subsection, we recall the definition of Euler systems. 

\begin{definition}\ 
\begin{itemize}
\item[(1)] Let $\Omega$ denote the set of fields $K$ in $\overline{\bQ}$ such that $K/\bQ$ is a finite abelian $p$-extension and 
$S_{\mathrm{ram}}(K/\bQ) \subset \cP$. 
Here $S_{\mathrm{ram}}(K/\bQ)$ is the set of primes at which $K/\bQ$ is ramified. 
\item[(2)] We say that $(c_{K})_{K \in \Omega} \in \prod_{K \in \Omega}H^{1}(G_{\bQ}, \bT_{K_{\infty}})$ is an Euler system of rank $1$ if, for any fields $K_{1} \subset K_{2}$ in $\Omega$, we have 
\[
\mathrm{Cor}_{K_{2}/K_{1}}(c_{K_{2}}) = \left(\prod_{\ell \in S_{\mathrm{ram}}(K_{2}/\bQ) \setminus S_{\mathrm{ram}}(K_{1}/\bQ)}P_{\ell}(\sigma_{\ell}^{-1}) \right) c_{K_{1}}. 
\] 
Here $\mathrm{Cor}_{K_{2}/K_{1}} \colon H^{1}(G_{\bQ}, \bT_{K_{2, \infty}}) \longrightarrow H^{1}(G_{\bQ}, \bT_{K_{1, \infty}})$ denotes the homomorphism  induced by $\bT_{K_{2, \infty}} \longrightarrow \bT_{K_{1, \infty}}$. 
Let $\mathrm{ES}_{1}(T)$ denote the set of Euler systems of rank $1$. 
\item[(3)] We say that $(c_{K})_{K \in \Omega} \in \prod_{K \in \Omega}\Lambda_{K_{\infty}}$ is an Euler system of rank $0$ if, for any fields $K_{1} \subset K_{2}$ in $\Omega$, we have 
\[
\pi_{K_{2}, K_{1}}(c_{K_{2}}) = \left(\prod_{\ell \in S_{\mathrm{ram}}(K_{2}/\bQ) \setminus S_{\mathrm{ram}}(K_{1}/\bQ)}P_{\ell}(\sigma_{\ell}^{-1}) \right) c_{K_{1}}. 
\] 
Here $\pi_{K_{2}, K_{1}} \colon \Lambda_{K_{2, \infty}} \longrightarrow \Lambda_{K_{1, \infty}}$ denotes the canonical projection map. 
Let $\mathrm{ES}_{0}(T)$ denote the set of Euler systems of rank $0$. 
\end{itemize}
\end{definition}

For any abelian field $K$ of conductor $d$, 
we denote by $\widetilde{\xi}_{K_{\infty}}$ the image of $\widetilde{\xi}_{\bQ(\mu_{dp^{\infty}})}$ in $\Lambda_{K_{\infty}}$. 
Then Proposition \ref{prop:rel} implies the following proposition. 
\begin{proposition}\label{prop:euler0}
We have $(\widetilde{\xi}_{K_{\infty}})_{K \in \Omega} \in \mathrm{ES}_{0}(T)$. 
\end{proposition}

Let $K \in \Omega$ be a field. Then, by Theorem \ref{pt},  for any integers $m \geq 1$ and $n \geq 0$, we have an exact sequence  
\begin{align*}
0 \longrightarrow \mathrm{Sel}(K_{n}, E[p^{m}]) \longrightarrow H^{1}_{\cF_{\mathrm{can}}}&(G_{\bQ}, T_{K_{n}}/p^{m}) 
\\
&\longrightarrow H^{1}_{/f}(G_{\bQ}, T_{K_{n}}/p^{m}) \longrightarrow \mathrm{Sel}(K_{n}, E[p^{m}])^{\vee}. 
\end{align*}
Here $\mathrm{Sel}(K_{n}, E[p^{m}])$ is the $p^{m}$-Selmer group of $E/K_{n}$ and 
\[
H^{1}_{\cF_{\mathrm{can}}}(G_{\bQ}, T_{K_{n}}/p^{m})  := \ker\left(H^{1}(G_{\bQ}, T_{K_{n}}/p^{m}) \longrightarrow \bigoplus_{\ell \neq p}H^{1}_{/\mathrm{ur}}(G_{\bQ_{\ell}}, T_{K_{n}}/p^{m}) \right). 
\]
We set 
\begin{align*}
H^{1}_{/f}(G_{\bQ_{p}}, \bT_{K_{\infty}}) &:= \varprojlim_{m,n} H^{1}_{/f}(G_{\bQ_{p}}, T_{K_{n}}/p^{m}), 
\\
\mathrm{Sel}(K_{\infty}, E[p^{\infty}]) &:= \varinjlim_{m,n}  \mathrm{Sel}(K_{n}, E[p^{m}]). 
\end{align*}
Since $\mathrm{Sel}(K_{\infty}, E[p^{\infty}])^{\vee}$ is a finitely generated torsion $\Lambda_{K_{\infty}}$-module, we have  
\[
\varprojlim_{m,n}  \mathrm{Sel}(K_{n}, E[p^{m}]) = 0. 
\] 
Moreover,  \cite[Proposition B.3.4]{R} implies 
\[
H^{1}(G_{\bQ}, \bT_{K_{\infty}}) = \varprojlim_{m,n} H^{1}_{\cF_{\mathrm{can}}}(G_{\bQ}, T_{K_{n}}/p^{m}). 
\]
Hence we get an exact sequence of $\Lambda_{K_{\infty}}$-modules 
\begin{align}\label{exact1}
0 \longrightarrow H^{1}(G_{\bQ}, \bT_{K_{\infty}}) 
\stackrel{\mathrm{loc}_{p}^{/f}}{\longrightarrow} H^{1}_{/f}(G_{\bQ_{p}}, \bT_{K_{\infty}}) 
\longrightarrow \mathrm{Sel}(K_{\infty}, E[p^{\infty}])^{\vee}. 
\end{align}
For each field $K \in \Omega$, we put 
\[
M_{K_{\infty}} := (\mathrm{loc}_{p}^{/f})^{-1}(H^{1}(G_{\bQ_{p}}, \bT_{K_{\infty}}) \cap \fL_{K_{\infty}}^{-1}(\Lambda_{K_{\infty}}) ), 
\]
and we obtain an injection 
\[
\fL \colon \mathrm{ES}_{1}(T) \cap \prod_{K \in \Omega}M_{K_{\infty}} \longhookrightarrow \mathrm{ES}_{0}(T); (c_{K})_{K \in \Omega} \mapsto (\mathrm{loc}_{p}^{/f}(\fL_{K_{\infty}}(c_{K})))_{K\in\Omega}. 
\]
Then Theorem \ref{thm:zeta} and the injectivity of $\mathrm{loc}_{p}^{/f} \circ \fL_{K_{\infty}}$ imply the following proposition. 

\begin{proposition}\label{prop:euler1}
There is an Euler system $z_{\xi} \in \mathrm{ES}_{1}(T) \cap \prod_{K \in \Omega}M_{K_{\infty}}$ such that $\fL(z_{\xi}) = (\widetilde{\xi}_{K_{\infty}})_{K \in \Omega}$. 
\end{proposition}

%


\subsection{Construction of $\kappa_{\xi,m,n}$}

Fix integers $m \geq 1$ and $n \geq 0$. 
First, we introduce the Kolyvagin derivative homomorphism (defined by Mazur and Rubin in \cite{MRkoly})
\[
\cD_{m,n}^{1} \colon {\rm ES}_{1}(T)  \longrightarrow  \mathrm{KS}_{1}(T_{\bQ_{n}}/p^{m}, \cF_{\rm can}). 
\]
Recall that $\bQ(d)$ is the maximal $p$-subextension of $\bQ(\mu_{d})$, and note that $\bQ_{n} = \bQ(p^{n+1})$. 
We fix a generator $g_{\ell}$ of $G_{\ell} = \Gal(\bQ(\ell)/\bQ)$ for each prime $\ell \in \cP_{1,0}$ and denote by $D_{\ell} \in \bZ[G_{\ell}]$ the Kolyvagin's derivative operator: 
\[
D_{\ell} := \sum_{i=0}^{\#G_{\ell} -1}i g_{\ell}^{i}. 
\]
For any integer $d \in \cN_{1,0}$, we also set $D_{d} := \prod_{\ell \mid d}D_{\ell} \in \bZ[\Gal(\bQ(d)/\bQ)]$. 

Let $c \in \mathrm{ES}_{1}(T)$ be an Euler system. 
For any integer $d \in \cN_{m,n}$, we denote by $c_{dp^{n+1}} \in H^{1}(G_{\bQ}, T_{\bQ(dp^{n+1})})$ the image of $c_{\bQ(d)} \in H^{1}(G_{\bQ}, \bT_{\bQ(d)})$. 
Then it is well-known that Euler system relations imply  
\[
\kappa(c)_{d, m,n} := D_{d}c_{dp^{n+1}} \bmod p^{m} \in H^{1}(G_{\bQ}, T_{\bQ(dp^{n+1})}/p^{m})^{\Gal(\bQ(d)/\bQ)}. 
\]
(see, for example, \cite[Lemma~4.4.2]{R}). 
Since we have an isomorphism 
\[
H^{1}(G_{\bQ}, T_{\bQ_{n}}/p^{m}) \stackrel{\sim}{\longrightarrow} H^{1}(G_{\bQ}, T_{\bQ(dp^{n+1})}/p^{m})^{\Gal(\bQ(d)/\bQ)}, 
\]
we can regard  $\kappa(c)_{d, m,n}$ as an element of $H^{1}(G_{\bQ}, T_{\bQ_{n}}/p^{m})$. 
The following theorem is proved by Mazur and Rubin in \cite[Appendix A]{MRkoly}.


\begin{theorem}\label{thm:euler-koly}
For any Euler system $c \in {\rm ES}_{1}(T)$, we have 
\[
\cD_{m,n}^{1}(c) := ({\kappa}(c)_{d, m, n})_{d \in \cN_{m,n}} \in \mathrm{KS}_{1}(T_{\bQ_{n}}/p^{m}, \cF_{\rm can}). 
\]
Hence we obtain the Kolyvagin derivative homomorphism 
\[
\cD_{m,n}^{1} \colon {\rm ES}_{1}(T) \longrightarrow  \mathrm{KS}_{1}(T_{\bQ_{n}}/p^{m}, \cF_{\rm can}). 
 \]
\end{theorem}

\begin{remark}
For any $\ell \in \cP_{m,n}$, we have 
\[
P_{\ell}(t) \equiv (t-1)^{2} \pmod{p^{m}}. 
\]
Hence ${\kappa}(c)_{d, m, n}$ coincides with $\kappa_{n}'$ defined in \cite[page 80, (33)]{MRkoly}. 
\end{remark}

Next let us construct a homomorphism 
\[
\cD_{m,n}^{0} \colon {\rm ES}_{0}(T) \longrightarrow \prod_{d \in \cN_{m,n}}R_{\bQ_{n}}/p^{m} \otimes_{\bZ} G_{d}. 
 \]
Let $c \in {\rm ES}_{0}(T)$ be an Euler system and take an integer $d \in \cN_{m,n}$. 
We denote by $c_{dp^{n+1}} \in R_{\bQ(dp^{n+1})}$ the image of $c_{\bQ(d)} \in \Lambda_{\bQ(d)}$. 

\begin{lemma}\label{lem:leading}
For any integer $d \in \cN_{m,n}$, we have 
\[
\delta(c)_{d,m,n} := D_{d}c_{dp^{n+1}} \bmod{p^{m}} \in (R_{\bQ(dp^{n+1})}/p^{m})^{\Gal(\bQ(d)/\bQ)} \stackrel{\sim}{\longleftarrow} R_{\bQ_{n}}/p^{m}.  
\]
Moreover, if we write $c_{dp^{n+1}} \bmod{p^{m}} = \sum_{\sigma \in \Gal(\bQ(d)/\bQ)}a_{\sigma} \sigma$, where $a_{\sigma} \in R_{\bQ_{n}}/p^{m}$, then we have 
\begin{align*}
\delta(c)_{d,m,n}  = (-1)^{\nu(d)}\sum_{\sigma \in \Gal(\bQ(d)/\bQ) }a_{\sigma} \prod_{\ell \mid d}\overline{\log}_{g_{\ell}}(\sigma).   
\end{align*}
Here 
\[
\overline{\log}_{g_{\ell}} \colon G_{\ell} \stackrel{\sim}{\longrightarrow} \bZ/(\ell-1) \longrightarrow  \bZ/p^{m}; \, g_{\ell}^{a} \mapsto a \bmod{p^{m}}
\] 
is the surjection induced by the discrete logarithm to the base $g_{\ell}$. 
\end{lemma}
\begin{proof}
The  assertion that 
\[
D_{d}c_{dp^{n+1}} \bmod{p^{m}} \in (R_{\bQ(dp^{n+1})}/p^{m})^{\Gal(\bQ(d)/\bQ)} 
\]
is well-known (see, for example, \cite[Lemma~4.4.2]{R}). 
Let us show the latter assertion. 
We write $d = \ell_{1} \cdots \ell_{t}$. 
We put 
\[
N_{\ell_{i}} :=  \sum_{\sigma \in \Gal(\bQ(\ell_{i})/\bQ)}\sigma \,\,\, \textrm{ and } \,\,\, X_{\ell_{i}} := g_{\ell_{i}} - 1. 
\]
Note that $D_{\ell_{i}} X_{\ell_{i}} = -N_{\ell_{i}}$ and $D_{\ell_{i}} X_{\ell_{i}}^{2} = 0$. 
Hence we have 
\begin{align*}
D_{d} \sum_{\sigma \in \Gal(\bQ(d)/\bQ)}a_{\sigma} \sigma 
&= \sum_{i_{1}=1}^{\#G_{\ell_{1}} - 1} \cdots \sum_{i_{t}=1}^{\#G_{\ell_{t}} - 1} 
a_{g_{\ell_{1}}^{i_{1}} \cdots g_{\ell_{t}}^{i_{t}}} D_{d} (1+X_{\ell_{1}})^{i_{1}} \cdots (1+X_{\ell_{t}})^{i_{t}}
\\
&= \sum_{i_{1}=1}^{\#G_{\ell_{1}} - 1} \cdots \sum_{i_{t}=1}^{\#G_{\ell_{t}} - 1} 
a_{g_{\ell_{1}}^{i_{1}} \cdots g_{\ell_{t}}^{i_{t}}} (1 - i_{1}N_{\ell_{1}}) \cdots (1- i_{t}N_{\ell_{t}})
\\
&=: \sum_{i=1}^{t}\sum_{j_{i} \in \{ 0, 1\}} b_{j_{1}, \ldots, j_{t}}N_{\ell_{1}}^{j_{1}} \cdots N_{\ell_{t}}^{j_{t}}. 
\end{align*}
Since 
\[
b_{1,\ldots,1} =  (-1)^{\nu(d)}\sum_{\sigma \in \Gal(\bQ(d)/\bQ) }a_{\sigma} \prod_{\ell \mid d}\overline{\log}_{g_{\ell}}(\sigma), 
\]
it suffices to show that $b_{j_{1}, \ldots, j_{t}} = 0$ for any $(j_{1}, \ldots, j_{t}) \neq (1, \ldots, 1)$. 
This follows from the facts that $X_{\ell_{i}}D_{d}c_{dp^{n+1}} \bmod{p^{m}} = 0$ and $X_{\ell_{i}}N_{\ell_{i}} = 0$ for any $1 \leq i \leq t$. 
In fact, since 
\[
0 = X_{\ell_{1}} \cdots X_{\ell_{t}}D_{d}c_{dp^{n+1}} \bmod{p^{m}} = b_{0, \ldots, 0}X_{\ell_{1}} \cdots X_{\ell_{t}}, 
\]
we have  $ b_{0, \ldots, 0} = 0$. Moreover, since 
\[
0 = X_{\ell_{2}} \cdots X_{\ell_{t}}D_{d}c_{dp^{n+1}} \bmod{p^{m}} = (b_{0,\ldots,0} + b_{1,0, \ldots, 0})X_{\ell_{2}} \cdots X_{\ell_{t}}, 
\] 
we have $b_{1,0, \ldots, 0} = 0$. 
Similary, we have $b_{0,1, \ldots, 0} = \cdots = b_{0, \ldots, 0,1} = 0$. Repeating this argument, we see that 
$b_{j_{1}, \ldots, j_{t}} = 0$ for any $(j_{1}, \ldots, j_{t}) \neq (1, \ldots, 1)$. 
\end{proof}


\begin{definition}
We define the homomorphism 
\[
\cD_{m,n}^{0} \colon {\rm ES}_{0}(T) \longrightarrow \prod_{d \in \cN_{m,n}}R_{\bQ_{n}}/p^{m} \otimes_{\bZ} G_{d}
\] 
by $\cD_{m,n}^{0}(c) := ({\delta}(c)_{d,m,n} )_{d \in \cN_{m,n}}$. 
\end{definition}

Recall that we have the isomorphism $\fL_{\bQ_{\infty}} \colon H^{1}_{/f}(G_{\bQ_{p}}, \bT_{\bQ_{\infty}}) \stackrel
{\sim}{\longrightarrow}  \Lambda_{\bQ_{\infty}}$ by Theorem \ref{thm:zeta}(ii). 
Since 
\[
H^{1}_{/f}(G_{\bQ_{p}}, \bT_{\bQ_{\infty}})  \otimes_{\Lambda_{\bQ_{\infty}}} R_{\bQ_{n}}/p^{m} \stackrel{\sim}{\longrightarrow} H^{1}_{/f}(G_{\bQ_{p}}, T_{\bQ_{n}}/p^{m}), 
\]
the isomorphism $\fL_{\bQ_{\infty}}$ induces an isomorphism 
\[
\fL_{\bQ_{n}, m} \colon H^{1}_{/f}(G_{\bQ_{p}}, T_{\bQ_{n}}/p^{m}) \stackrel{\sim}{\longrightarrow} R_{\bQ_{n}}/p^{m}, 
\]
and hence we obtain a homomorphism 
\[
\fL_{\bQ_{n}, m}  \colon \mathrm{KS}_{1}(T_{\bQ_{n}}/p^{m}, \cF_{\mathrm{can}})  \longrightarrow  \prod_{d \in \cN_{m,n}}R_{\bQ_{n}}/p^{m} \otimes_{\bZ} G_{d}. 
\]
By construction, we have the following proposition. 
\begin{proposition}\label{prop:diag}
The diagram
\[
\xymatrix{
\mathrm{ES}_{1}(T)  \cap \prod_{K \in \Omega} M_{K_{\infty}} \ar@{^{(}->}[rr]^-{\fL} \ar[d]^-{\cD^{1}_{m,n}} && \mathrm{ES}_{0}(T)  \ar[d]^-{\cD^{0}_{m,n}}
\\
\mathrm{KS}_{1}(T_{\bQ_{n}}/p^{m}, \cF_{\mathrm{can}})  \ar[rr]^-{\fL_{\bQ_{n},m}} &&  \prod_{d \in \cN_{m,n}}R_{\bQ_{n}}/p^{m} \otimes_{\bZ} G_{d}
}
\]
commutes. 
\end{proposition}

\begin{theorem}\label{thm:main}
There is a Kolyvagin system $\kappa_{\xi,m,n} \in \mathrm{KS}_{0}(T_{\bQ_{n}}/p^{m}, \cF_{\mathrm{cl}})$ satisfying $\delta(\kappa_{\xi,m,n}) = \cD^{0}_{m,n}((\widetilde{\xi}_{K_{\infty}})_{K \in \Omega})$. 
\end{theorem}
\begin{proof}
 Let $z_{\xi} \in \mathrm{ES}_{1}(T)$ be the Euler system defined in Proposition \ref{prop:euler1}. Note that $\fL(c_{\xi}) = (\widetilde{\xi}_{K_{\infty}})_{K \in \Omega}$. 
 We define 
 \[
 \kappa_{\xi,m,n} := \mathrm{\Phi} \circ \cD_{m,n}^{1}(c_{\xi}). 
 \]
 Here $\Phi \colon \mathrm{KS}_{1}(T, \cF_{\mathrm{can}}) \longrightarrow \mathrm{KS}_{0}(T, \cF_{\mathrm{cl}})$ is the homomorphism associated with the isomorphism $\fL_{\bQ_{n}, m} \colon H^{1}_{/f}(G_{\bQ_{p}}, T_{\bQ_{n}}/p^{m}) \stackrel{\sim}{\longrightarrow} R_{\bQ_{n}}/p^{m}$ (see \S\ref{sec:rank reduction}). 
 The commutative diagram \eqref{diag1} shows that $\delta \circ \Phi = \fL_{\bQ_{n}, m}$. Hence Proposition \ref{prop:diag} implies 
 \begin{align*}
 \delta(\kappa_{\xi,m,n}) &= \delta \circ \Phi \circ \cD_{m,n}^{1}(z_{\xi}) 
 \\
 &= \fL_{\bQ_{n}, m} \circ \cD_{m,n}^{1}(z_{\xi}) 
 \\
 &=  \cD_{m,n}^{0} \circ \fL(z_{\xi}) 
 \\
 &= \cD^{0}_{m,n}((\widetilde{\xi}_{K_{\infty}})_{K \in \Omega}). 
 \end{align*}
\end{proof}

\begin{remark}
The Kolyvagin system $\kappa_{\xi,m,n}$ constructed in Theorem \ref{thm:main} is a natural extension of a family of cohomology classes constructed by Kurihara in \cite{Kur14b}  (see also \cite{Kur14a}). 
More precisely, for any ``admissible''  pair $(d, \ell) \in \cM_{m,n}$, Kurihara constructed a  cohomology class $\kappa_{d,\ell}$ such that  it satisfies the relations appeared in the definition of Kolyvagin system of rank $0$ and that it relates to modular symbols via the map $\delta$. In our construction, we do not need to impose that the pair $(d, \ell) \in \cN_{m,n} \times \cP_{m,n}$ is admissible. 
\end{remark}

\subsection{Properties of $\kappa_{\xi, m,n}$}

Recall that the Iwasawa main conjecture for $E/\bQ$ says that  
\[
\widetilde{\xi}_{\bQ_{\infty}} \Lambda_{\bQ_{\infty}} = \mathrm{char}_{\Lambda_{\bQ_{\infty}}}(\mathrm{Sel}(\bQ_{\infty}, E[p^{\infty}])^{\vee}). 
\]

\begin{proposition}\label{prop:equiv}
The following are equivalent. 
\begin{itemize}
\item[(1)] The Kolyvagin system $\kappa_{\xi, m,n} \in \mathrm{KS}_{0}(T_{\bQ_{n}}/p^{m}, \cF_{\mathrm{cl}})$ is a basis for some  $m \geq 1$ and $n \geq 0$. 
\item[(2)] The Kolyvagin system $\kappa_{\xi, m,n} \in \mathrm{KS}_{0}(T_{\bQ_{n}}/p^{m}, \cF_{\mathrm{cl}})$ is a basis for any  $m \geq 1$ and $n \geq 0$. 
\item[(3)] There is an integer $d \in \cN_{1,0}$ satisfying $\delta(\kappa_{\xi, 1, 0})_{d} \neq 0$. 
\item[(4)] The Iwasawa main conjecture for $E/\bQ$ holds true. 
\end{itemize}
\end{proposition}
\begin{proof}
We put 
\[
\mathrm{KS}_{0}(\bT_{\bQ_{\infty}}, \cF_{\mathrm{cl}}) := \varprojlim_{m,n} \mathrm{KS}_{0}(T_{\bQ_{n}}/p^{m}, \cF_{\mathrm{cl}}). 
\]
Then Theorem \ref{thm:koly0} and \cite[Lemma 3.25]{sakamoto} (see \cite[Theorem 6.3]{sakamoto-koly0}) show that 
the canonical map $\mathrm{KS}_{0}(\bT_{\bQ_{\infty}}, \cF_{\mathrm{cl}})  \longrightarrow \mathrm{KS}_{0}(T_{\bQ_{n}}/p^{m}, \cF_{\mathrm{cl}})$ is surjective and the $\Lambda_{\bQ_{\infty}}$-module $\mathrm{KS}_{0}(\bT_{\bQ_{\infty}}, \cF_{\mathrm{cl}})$ is free of rank $1$.  
By construction, 
\[
\kappa_{\xi} := (\kappa_{\xi, m,n})_{m\geq 1,n \geq 0} \in \mathrm{KS}_{0}(\bT_{\bQ_{\infty}}, \cF_{\mathrm{cl}}). 
\]
Since $\delta \colon \mathrm{KS}_{0}(E[p], \cF_{\mathrm{cl}}) \longrightarrow \prod_{d \in \cN_{1,0}}\bF_{p} \otimes_{\bZ} G_{d}$ is injective by Corollary \ref{cor:inj},  claims (1), (2) and (3) are equivalent, and it  suffices to show that  claim (4) is equivalent to that $\kappa_{\xi}$ is a basis. We have the canonical homomorphism 
\[
\delta_{1} \colon \mathrm{KS}_{0}(\bT_{\bQ_{\infty}}, \cF_{\mathrm{cl}}) 
\longrightarrow  \Lambda_{\bQ_{\infty}}; \, (\kappa_{m,n})_{m \geq 1, n \geq 0} \mapsto \varprojlim_{m,n} \delta(\kappa_{m,n})_{1}. 
\]
By Theorem \ref{thm:main}, we have 
\[
\delta_{1}(\kappa_{\xi}) = \varprojlim_{m,n}  \widetilde{\delta}( (\widetilde{\xi}_{K_{\infty}})_{K\in\Omega})_{1,m,n} 
= \varprojlim_{m,n} \widetilde{\xi}_{p^{n+1}} \bmod p^{m} = \widetilde{\xi}_{\bQ_{\infty}}. 
\]
Let $\kappa \in \mathrm{KS}_{0}(\bT_{\bQ_{\infty}}, \cF_{\mathrm{cl}})$ be a basis and write $\kappa_{\xi} = a \kappa$ for some $a \in \Lambda_{\bQ_{\infty}}$. 
Then, by Theorem \ref{thm:koly0} (see \cite[Theorem 6.4]{sakamoto-koly0}), we have 
\[
\widetilde{\xi}_{\bQ_{\infty}}\Lambda_{\bQ_{\infty}} = a \delta_{1}(\kappa) = a \cdot \mathrm{char}_{\Lambda_{\bQ_{\infty}}}(\mathrm{Sel}(\bQ_{\infty}. E[p^{\infty}])^{\vee}).  
\]
Since the characteristic ideal $\mathrm{char}_{\Lambda_{\bQ_{\infty}}}(\mathrm{Sel}(\bQ_{\infty}. E[p^{\infty}])^{\vee})$ is non-zero, 
 claim (4) is equivalent to that $a$ is unit, i.e.,  $\kappa_{\xi}$ is a basis. 
 \end{proof}


\section{Main results}\label{sec:main}

\subsection{Proof of Theorem \ref{thm:1}}
First, let us discuss the relation between $\delta(\kappa_{\xi, 1, 0})_{d}$ and $\widetilde{\delta}_{d}$. 
As in \S\ref{sec:intro}, for each prime $\ell \in \cP_{1,0}$, we fix a generator $h_{\ell} \in \Gal(\bQ(\mu_{\ell})/\bQ)$, and 
it naturally induces the surjection 
\[
\overline{\log}_{h_{\ell}} \colon \Gal(\bQ(\mu_{\ell})/\bQ) \stackrel{\sim}{\longrightarrow} \bZ/(\ell-1) \longrightarrow \bF_{p}; \, h_{\ell}^{a} \mapsto a \bmod{p}. 
\]
Recall that, for any integer $d \in \cN_{1,0}$, the analytic quantity  $\widetilde{\delta}_{d} \in \bF_{p}$ is defined by  
\[
\widetilde{\delta}_{d} := \sum_{\substack{a=1 \\ (a,d)=1}}^{d}\frac{\mathrm{Re}([a/d])}{\Omega_{E}^{+}} \cdot \prod_{\ell \mid d}\overline{\log}_{h_{\ell}}(\sigma_{a}).  
\]
We put $e_{d} := \# \Gal(\bQ(\mu_{d})/\bQ(d))$. Since $p \nmid e_{d}$, 
we see that $\widetilde{\delta}_{d} = 0$ if and only if 
\[
e_{d}^{\nu(d)}\widetilde{\delta}_{d} = \sum_{\substack{a=1 \\ (a,d)=1}}^{d}\frac{\mathrm{Re}([a/d])}{\Omega_{E}^{+}} \cdot \prod_{\ell \mid d}\overline{\log}_{h_{\ell}}(\sigma_{a}^{e_{d}}) = 0. 
\]
Let $\widetilde{\theta}_{d} = \sum_{\sigma \in \Gal(\bQ(d)/\bQ)}a_{\sigma} \sigma$ denote the image of $\widetilde{\theta}_{\bQ_{(\mu_{d})}}$ in $\bZ_{p}[\Gal(\bQ(d)/\bQ)]$ (see \S\ref{sec:modular} for the definition of $\widetilde{\theta}_{\bQ_{(\mu_{d})}}$). 
Assume for simplicity that the image of $h_{\ell}^{e_{d}}$ is the fixed generator $g_{\ell} \in \Gal(\bQ(\ell)/\bQ)$. 
Recall that we have the surjection  
\[
\overline{\log}_{g_{\ell}} \colon \Gal(\bQ(d)/\bQ) \stackrel{\sim}{\longrightarrow} \bZ/(\ell-1) \longrightarrow \bF_{p}. 
\]
Since $\sigma_{a} = \sigma_{b}$ in $\Gal(\bQ(d)/\bQ)$ if $\sigma_{a}^{e_{d}} = \sigma_{b}^{e_{d}}$, we see that 
\[
e_{d}^{\nu(d)}\widetilde{\delta}_{d} = \sum_{\sigma \in \Gal(\bQ(d)/\bQ)}a_{\sigma} \cdot \prod_{\ell \mid d}\overline{\log}_{g_{\ell}}(\sigma). 
\]
Since we have 
\[
D_{d}\widetilde{\theta}_{d} \bmod{p}  = (-1)^{\nu(d)} \left( \sum_{\sigma \in \Gal(\bQ(d)/\bQ)}a_{\sigma} \cdot \prod_{\ell \mid d}\overline{\log}_{g_{\ell}}(\sigma)\right) N_{d}
\]
by Lemma \ref{lem:leading}, we obtain the following lemma. 
\begin{lemma}\label{lem:non-vanishing}
For any integer $d \in \cN_{1,0}$, the following are equivalent. 
\begin{itemize}
\item[(1)] $\widetilde{\delta}_{d}  \neq 0$. 
\item[(2)] $D_{d}\widetilde{\theta}_{d} \bmod{p} \neq 0$. 
\end{itemize}
\end{lemma}

\begin{lemma}\label{lem:non-vanishing2}
For any integer $d \in \cN_{1,0}$, the following are equivalent. 
\begin{itemize}
\item[(1)] $\widetilde{\delta}_{d}  \neq 0$. 
\item[(2)] $\delta(\kappa_{\xi,1,0})_{d} \neq 0$. 
\end{itemize}
\end{lemma}
\begin{proof}
Since any prime $\ell \in \cP_{1,0}$ is  congruent to 1 modulo $p$, the relation $\delta(\kappa_{\xi,1,0}) = \cD^{0}_{1,0}((\widetilde{\xi}_{K_{\infty}})_{K \in \Omega})$ in Theorem \ref{thm:main} shows that 
$\delta(\kappa_{\xi,1,0})_{d} \neq 0$ if and only if  $D_{d}\widetilde{\vartheta}_{d} \bmod{p} \neq 0$. 
Hence this lemma follows from Lemma \ref{lem:non-vanishing} and Remark \ref{rem:bottom}. 
\end{proof}

\begin{corollary}[Theorem \ref{thm:1}]\label{cor:conj1}
Conjecture \ref{conj:1} holds true, that is, there is an integer $d \in \cN_{1,0}$ satisfying $\widetilde{\delta}_{d} \neq 0$ if and only if 
the Iwasawa main conjecture for $E/\bQ$ holds true. 
\end{corollary}

\begin{proof}
This corollary follows from Proposition \ref{prop:equiv}  and Lemma \ref{lem:non-vanishing2}. 
\end{proof}

\subsection{Proof of Theorem \ref{thm:main1}}\label{sec:proof}

In this subsection, we give a proof of Theorem \ref{thm:main1}. 
Recall that an integer $d \in \cN_{1,0}$ is $\delta$-minimal if $\widetilde{\delta}_{d} \neq  0$ and $\widetilde{\delta}_{e} = 0$ for any positive proper divisor $e$ of $d$. 
Note that the existence of a $\delta$-minimal integer implies  that the Kolyvagin system $\kappa_{\xi,1,0}$ is a basis of $\mathrm{KS}_{0}(E[p], \cF_{\mathrm{cl}})$ by Proposition \ref{prop:equiv} and Corollary \ref{cor:conj1}. 

\begin{lemma}\label{lem:delta-vanish}
Let $d \in \cN_{1,0}$ be an integer.  
Then the following are equivalent. 
\begin{itemize}
\item[(1)] $\widetilde{\delta}_{d}  \neq 0$. 
\item[(2)] $H^{1}_{\cF_{\mathrm{cl}}(d)}(G_{\bQ}, E[p]) = 0$. 
\end{itemize}
\end{lemma}
\begin{proof}
By Theorem \ref{thm:koly0}, we have 
\[
\bF_{p} \cdot \delta(\kappa_{\xi,1,0})_{d} = \mathrm{Fitt}_{\bF_{p}}^{0}(H^{1}_{\cF_{\mathrm{cl}}(d)}(G_{\bQ}, E[p])^{\vee}). 
\]
Hence this lemma follows from  Lemma  \ref{lem:non-vanishing2}. 
\end{proof}

\begin{remark}\label{rem:inj}
The injectivity of the homomorphism \eqref{map} (proved by Kurihara) follows immediately from Lemma \ref{lem:delta-vanish}. In fact, we have 
\[
\ker\left(\mathrm{Sel}(\bQ, E[p]) \stackrel{\eqref{map}}{\longrightarrow} \bigoplus_{\ell \mid d} E(\bQ_{\ell}) \otimes \bF_{p} \right) = 
H^{1}_{(\cF_{\mathrm{cl}})_{d}}(G_{\bQ}, E[p]) \subset H^{1}_{\cF_{\mathrm{cl}}(d)}(G_{\bQ}, E[p]). 
\]
\end{remark}

For any integer $d \in \cN_{1,0}$, we set 
\[
\lambda(d) := \dim_{\bF_{p}}(H^{1}_{\cF_{\mathrm{cl}}(d)}(G_{\bQ}, E[p])). 
\]

\begin{lemma}\label{lem:rank}
Let $d \in \cN_{1,0}$ be an integer and $\ell \in \cP_{1,0}$ a prime with $\ell \nmid d$. 
\begin{itemize}
\item[(1)] If $H^{1}_{\cF_{\mathrm{cl}}(d)}(G_{\bQ}, E[p]) \neq H^{1}_{(\cF_{\mathrm{cl}})_{\ell}(d)}(G_{\bQ}, E[p])$, then 
$\lambda(d\ell) = \lambda(d) - 1$. 
\item[(2)] If $H^{1}_{\cF_{\mathrm{cl}}(d)}(G_{\bQ}, E[p]) = H^{1}_{(\cF_{\mathrm{cl}})_{\ell}(d)}(G_{\bQ}, E[p])$, then $\lambda(d) \leq \lambda(d\ell)$. 
\end{itemize}
In particular, $\lambda(d) \geq \lambda(1) - \nu(d)$. 
\end{lemma}
\begin{proof}
If $H^{1}_{\cF_{\mathrm{cl}}(d)}(G_{\bQ}, E[p]) \neq H^{1}_{(\cF_{\mathrm{cl}})_{\ell}(d)}(G_{\bQ}, E[p])$, then the localization map 
\[
H^{1}_{\cF_{\mathrm{cl}}(d)}(G_{\bQ}, E[p]) \longrightarrow H^{1}_{\rm ur}(G_{\bQ_{\ell}}, E[p])
\] 
is non-zero. 
Since $\cF_{\mathrm{cl}}(d)^{*} = \cF_{\mathrm{cl}}(d)$, claim (1) follows from \cite[Lemma 4.1.7 (iv)]{MRkoly}. 
Claim (2) is trivial since 
\[
H^{1}_{\cF_{\mathrm{cl}}(d)}(G_{\bQ}, E[p]) = H^{1}_{(\cF_{\mathrm{cl}})_{\ell}(d)}(G_{\bQ}, E[p]) \subset H^{1}_{\cF_{\mathrm{cl}}(d \ell)}(G_{\bQ}, E[p]).
\] 
\end{proof}

\begin{proposition}\label{prop:min-div}
Let $d \in \cN_{1,0}$ be an integer satisfying $H^{1}_{\cF_{\mathrm{cl}}(d)}(G_{\bQ}, E[p]) = 0$. 
Then there is a positive divisor $e$  of $d$ such that $\nu(e) = \lambda(1)$ and $\lambda(e)= 0$. 
\end{proposition}
\begin{proof} 
When $\lambda(1) = 0$, one can take $d = 1$. 
Hence we may assume that $\lambda(1) > 0$. 
If $H^{1}_{\cF_{\mathrm{cl}}}(G_{\bQ}, E[p]) = H^{1}_{(\cF_{\mathrm{cl}})_{\ell}}(G_{\bQ}, E[p])$ for any prime  $\ell \mid d$, then 
\begin{align*}
H^{1}_{\cF_{\mathrm{cl}}}(G_{\bQ}, E[p]) &= \bigcap_{\ell \mid d}H^{1}_{(\cF_{\mathrm{cl}})_{\ell}}(G_{\bQ}, E[p])
\\
&= H^{1}_{(\cF_{\mathrm{cl}})_{d}}(G_{\bQ}, E[p]) 
\\
&\subset H^{1}_{\cF_{\mathrm{cl}}(d)}(G_{\bQ}, E[p])
\\
&=0. 
\end{align*}
However, since we assume $\lambda(1) > 0$, we conclude that there is a prime $\ell_{1} \mid d$ such that  
\[
H^{1}_{\cF_{\mathrm{cl}}}(G_{\bQ}, E[p]) \neq H^{1}_{(\cF_{\mathrm{cl}})_{\ell_{1}}}(G_{\bQ}, E[p]). 
\]
Hence Lemma \ref{lem:rank}  implies   $\lambda(\ell_{1}) = \lambda(1) -1$. 
If $\lambda(1) = 1$, then $\ell_{1}$ is a desired divisor of $d$. 
Suppose that $\lambda(1) > 1$. Since  
\[
H^{1}_{(\cF_{\mathrm{cl}})_{d/\ell_{1}}(\ell_{1})}(G_{\bQ}, E[p]) \subset H^{1}_{\cF_{\mathrm{cl}}(d)}(G_{\bQ}, E[p])=0,
\]  
the same argument shows that there is a  prime $\ell_{2} \mid d/\ell_{1}$ satisfying  
\[
H^{1}_{\cF_{\mathrm{cl}}(\ell_{1})}(G_{\bQ}, E[p]) \neq H^{1}_{(\cF_{\mathrm{cl}})_{\ell_{2}}(\ell_{1})}(G_{\bQ}, E[p]). 
\]
Then $\lambda(\ell_{1}\ell_{2}) = \lambda(\ell_{1}) - 1$ by Lemma \ref{lem:rank}. 
 By repeating this argument, we obtain a sequence $\ell_{1}, \ldots, \ell_{\lambda(1)}$ of prime divisors of $d$ such that 
  $\lambda(\ell_{1}) = \lambda(1) - 1$ and  $\lambda(\ell_{1} \cdots \ell_{i+1}) = \lambda(\ell_{1} \cdots \ell_{i}) - 1$ for any $1 \leq i < \lambda(1)$.  Then $e := \ell_{1} \cdots \ell_{\lambda(1)}$ is a desired divisor of $d$. 
\end{proof}

\begin{theorem}[Theorem \ref{thm:main1}]\label{thm:mainmain}
For any $\delta$-minimal integer $d \in \cN_{1,0}$, we have 
\[
\dim_{\bF_{p}}(\mathrm{Sel}(\bQ, E[p])) = \nu(d). 
\]
\end{theorem}
\begin{proof}
Let  $d \in \cN_{1,0}$ be a $\delta$-minimal integer. 
Then  $H^{1}_{\cF_{\mathrm{cl}}(d)}(G_{\bQ}, E[p]) = 0$ by Lemma \ref{lem:delta-vanish}. 
Hence Proposition \ref{prop:min-div} shows that there is a positive divisor $e$ of $d$ such that $\nu(e) = \lambda(1)$ and $\lambda(e) = 0$. 
Then Lemma \ref{lem:non-vanishing2} implies $\widetilde{\delta}_{e} \neq 0$, and we have $d = e$ by the definition of the $\delta$-minimality. 
Therefore, we obtain $\nu(d) = \nu(e) = \lambda(1)$. 
%
%
\end{proof}

\begin{remark}\label{rem:self}
In the multiplicative group case, under the validity of the analogue of  Lemma \ref{lem:rank}, 
one can show that the analogue of Theorem \ref{thm:main1} (\cite[Conjecture 2]{Kur14b}) holds true. 
However, as mentioned in Remark \ref{rem:counter},  there is a counter-example of the analogue of Theorem \ref{thm:main1}. 
This shows that the analogue of  Lemma \ref{lem:rank} does not hold  in general. 
In the proof of Lemma \ref{lem:rank}, we use crucially the fact that the Selmer structure $\cF_{\mathrm{cl}}$ is self-dual, and hence one can say that the self-duality of the Selmer structure $\cF_{\mathrm{cl}}$ is one of the most important ingredients in order to prove Theorem \ref{thm:main1}. 
\end{remark}

Let $\kappa_{\xi,1,0}  = (\kappa_{d,\ell})_{(d,\ell) \in \cM_{1,0}} \in \mathrm{KS}_{0}(E[p], \cF_{\mathrm{cl}})$ be the Kolyvagin system constructed in Theorem \ref{thm:main}. 
By using  the fixed generator $g_{\ell} \in G_{\ell}$, we regard $G_{\ell}$ as $\bZ/\#G_{\ell}$, and hence 
one can regard $\kappa_{d,\ell} \in H^{1}_{\cF_{\mathrm{cl}}^{\ell}(d)}(G_{\bQ}, E[p])$. 
As discussed by Kurihara in \cite[Theorem 3(2)]{Kur14b}, 
by using Theorem \ref{thm:mainmain}, one can construct a basis of the $p$-Selmer group $\mathrm{Sel}(\bQ, E[p])$ from the Kolyvagin system $\kappa_{\xi,1,0}$. 

\begin{corollary}\label{cor:basis}
For any $\delta$-minimal integer $d=\ell_{1} \cdots \ell_{t} \in \cN_{1,0}$, the set $\{\kappa_{d/\ell_{i},\ell_{i}} \mid 1 \leq i \leq t\}$ is a basis of $\mathrm{Sel}(\bQ, E[p])$. 
\end{corollary}

\begin{proof}
Applying Theorem \ref{pt} with $\cF_{1} = (\cF_{\mathrm{cl}})_{d}$ and $\cF_{2} = \cF_{\mathrm{cl}}$, we obtain an exact sequence 
\begin{align*}
0 \longrightarrow H^{1}_{(\cF_{\mathrm{cl}})_{d}}(G_{\bQ}, E[p]) \longrightarrow \mathrm{Sel}(\bQ, E[p])
\longrightarrow \bigoplus_{\ell \mid d}H^{1}_{\mathrm{ur}}(G_{\bQ}, E[p])  
\\
\longrightarrow H^{1}_{\cF_{\mathrm{cl}}^{d}}(G_{\bQ}, E[p])^{\vee} 
\longrightarrow \mathrm{Sel}(\bQ, E[p])^{\vee} 
\longrightarrow 0. 
\end{align*}
Lemma \ref{lem:delta-vanish} and Theorem \ref{thm:mainmain} show that $H^{1}_{\cF_{\mathrm{cl}}^{d}}(G_{\bQ}, E[p]) =  \mathrm{Sel}(\bQ, E[p])$, and we have an isomorphism 
\[
\bigoplus_{\ell \mid d} \varphi^{\mathrm{fs}}_{\ell} \colon \mathrm{Sel}(\bQ, E[p]) 
\stackrel{\sim}{\longrightarrow} \bigoplus_{\ell \mid d}H^{1}_{\mathrm{ur}}(G_{\bQ}, E[p]) 
\stackrel{\sim}{\longrightarrow} \bF_{p}^{t}. 
\]
In particular, $\kappa_{d/\ell_{i},\ell_{i}}  \in \mathrm{Sel}(\bQ, E[p])$ for any integer $1 \leq i\leq t$.  
Take an integer $1 \leq i\leq t$. 
Since $H^{1}_{\cF_{\mathrm{cl}}^{\ell_{i}}(d/\ell_{i})}(G_{\bQ}, E[p])  \subset \mathrm{Sel}(\bQ, E[p])$, we have 
\[
H^{1}_{\cF_{\mathrm{cl}}^{\ell_{i}}(d/\ell_{i})}(G_{\bQ}, E[p]) = H^{1}_{\cF_{\mathrm{cl}}^{\ell_{i}}(d/\ell_{i})}(G_{\bQ}, E[p]) \cap H^{1}_{\cF_{\mathrm{cl}}}(G_{\bQ}, E[p])= H^{1}_{(\cF_{\mathrm{cl}})_{d/\ell_{i}}}(G_{\bQ}, E[p]). 
\]
Since $\kappa_{d/\ell_{i},\ell_{i}} \in H^{1}_{(\cF_{\mathrm{cl}})_{d/\ell_{i}}}(G_{\bQ}, E[p])$, we have $\varphi^{\mathrm{fs}}_{\ell_{j}}(\kappa_{d/\ell_{i},\ell_{i}}) = 0$ for any $j \neq i$. 
The $\delta$-minimality of $d$ and Lemma \ref{lem:non-vanishing2} imply that $\varphi^{\mathrm{fs}}_{\ell_{i}}(\kappa_{d/\ell_{i},\ell_{i}}) 
= -\delta(\kappa_{\xi,1,0})_{d} \neq 0$. This shows that the set $\{\kappa_{d/\ell_{i},\ell_{i}} \mid 1 \leq i \leq t\}$ is a basis of $\mathrm{Sel}(\bQ, E[p])$. 
\end{proof}

\appendix

\section{Remarks on $p=3$}\label{sec:appendix}

The assumption that $p>3$ is one of the standard hypotheses of the theory of Kolyvagin systems (see the hypothesis (H.4) in the page 27 of \cite{MRkoly}). 
In this appendix, we explain that Theorem \ref{thm:koly0} and Proposition \ref{prop:reg-isom} are valid even when $p=3$. 
We  note that, in the theory of Stark systems, the assumption that $p > 3$ is not needed (see \cite[Hypothesis 3.12]{sakamoto}). Hence one can use all results in \cite{sakamoto} even if $p=3$.

In this appendix, we consider the following situation. 
\begin{itemize}
\item $R$ is a zero-dimensional Gorenstein local ring  with finite residue field $\bF$ such that $p^{n}R=0$ and $\mathrm{char}(\bF) = 3$. 
\item $T$ is a free $R$-module of finite rank with a continuous $G_{\bQ}$-action satisfying the following: 
\begin{itemize}
\item $T \otimes_{R} \bF$ is an irreducible $\bF[G_{\bQ}]$-module. 
\item There is a rational prime $\ell \not\in S_{\mathrm{ram}}(T) \cup \{3\}$ such that $T/(\mathrm{Fr}_{\ell}-1)T \cong R$ and $\ell \equiv 1 \bmod{3^{n}}$. 
\item $H^{1}(\Gal(\bQ(\mu_{3^{n}}, T)/\bQ), T \otimes_{R}\bF) = 0$. Here $\bQ(\mu_{3^{n}}, T)$ is the filed corresponds to the kernel of $G_{\bQ(\mu_{3^{n}})} \longrightarrow \mathrm{Aut}(T)$. 
\item $T$ is residually self-dual, i.e., there is a $G_{\bQ}$-isomorphism $T \otimes_{R} \bF \cong (T \otimes_{R} \bF)^{\vee}(1)$. 
\end{itemize}
\end{itemize}
We put 
\begin{itemize}
\item $\overline{T} := T \otimes_{R} \bF$, 
\item $\cP := \{\ell \not\in S_{\mathrm{ram}}(T) \cup \{3\} \mid T/(\mathrm{Fr}_{\ell}-1)T \cong R, \, \ell \equiv 1 \bmod{3^{n}}\}$, 
\item $\cN$ denotes the set of square-free products in $\cP$. 
\end{itemize}

\subsection{Application of the Chebotarev density theorem}
As mentioned in the beginning of \cite[\S3.6]{MRkoly}, in the theory of Kolyvagin systems, the assumption that $p>3$ is only used for choosing useful primes. 
In this subsection, we prove a slightly weaker result than \cite[Proposition 3.6.1]{MRkoly} when $p=3$.

\begin{lemma}\label{lem:coset}
Let $a>0$ be an integer. 
Let $G$  be a group and $\varphi_{1}, \varphi_{2}, \varphi_{3}, \varphi_{4} \in \Hom(G, \bF_{3}^{a}) \setminus \{0\}$. 
Suppose that 
\[
\dim_{\bF_{3}} \left( \bF_{3}\varphi_{1} + \bF_{3} \varphi_{2} + \bF_{3} \varphi_{3} + \bF_{3} \varphi_{4} \right) \geq 3. 
\]
Then, for any $g_{1}, g_{2}, g_{3}, g_{4} \in G$, we have 
\[
\bigcup_{i=1}^{4}g_{i} \ker(\varphi_{i}) \neq G. 
\]
\end{lemma}


\begin{proof}
Put $\varphi_{i, j} := \mathrm{pr}_{j} \circ \varphi_{i} \colon G \longrightarrow  \bF_{3}$. Then  
\begin{align*}
\bigcup_{i=1}^{4}g_{i} \ker(\varphi_{i}) &= \bigcup_{i=1}^{4}  \bigcap_{ ( j_{1}, j_{2}, j_{3}, j_{4}) \in \{1, \ldots, a\}^{4} }  g_{i}\ker(\varphi_{i, j_{i}}) 
\\
&\subset  \bigcap_{ ( j_{1}, j_{2}, j_{3}, j_{4}) \in \{1, \ldots, a\}^{4} }  \bigcup_{i=1}^{4}  g_{i}\ker(\varphi_{i, j_{i}}). 
\end{align*}
Hence we may assume that $a=1$. 

Suppose that $\dim_{\bF_{3}} \left( \bF_{3}\varphi_{1} + \bF_{3} \varphi_{2} + \bF_{3} \varphi_{3} + \bF_{3} \varphi_{4} \right) = 4$. 
Since the kernel of the surjection 
\[
G \longrightarrow \bF_{3}^{4}; \, g \mapsto (\varphi_{1}(g), \varphi_{2}(g), \varphi_{3}(g), \varphi_{4}(g))
\] 
is contained in $\ker(\varphi_{i})$ for any $1 \leq i \leq 4$,  
 we may assume that $G = \bF_{3}^{4}$ and $\varphi_{i} = \mathrm{pr}_{i}$ for each $1 \leq i\leq 4$. 
In this case, an explicit calculation shows that 
\begin{align*}
G \setminus (g_{1} \ker(\varphi_{1}) &\cup g_{2} \ker(\varphi_{2}) \cup g_{3} \ker(\varphi_{3})  \cup g_{4} \ker(\varphi_{4}) ) 
\\
&= \{(h_{1}, h_{2}, h_{3}, h_{4}) \in \bF_{3}^{4} \mid \mathrm{pr}_{i}(g_{i}) \neq h_{i} \textrm{ for any $1 \leq i \leq 4$}\} \neq \emptyset. 
\end{align*}

Suppose that $\dim_{\bF_{3}} \left( \bF_{3}\varphi_{1} + \bF_{3} \varphi_{2} + \bF_{3} \varphi_{3} + \bF_{3} \varphi_{4} \right) = 3$. We may then assume that $\varphi_{4} \in \bF_{3}\varphi_{1} + \bF_{3} \varphi_{2} + \bF_{3} \varphi_{3}$. 
Moreover, since the kernel of the surjection 
\[
G \longrightarrow \bF_{3}^{3}; \, g \mapsto (\varphi_{1}(g), \varphi_{2}(g), \varphi_{3}(g))
\] 
is contained in $\ker(\varphi_{i})$ for any $1 \leq i \leq 4$, 
 we may also assume that $G = \bF_{3}^{3}$ and $\varphi_{i} = \mathrm{pr}_{i}$ for each $1 \leq i\leq 3$. 
 Then we have 
 \begin{align*}
G \setminus (g_{1} \ker(\varphi_{1}) \cup g_{2} &\ker(\varphi_{2}) \cup g_{3} \ker(\varphi_{3}) ) 
\\
&= \{(h_{1}, h_{2}, h_{3}) \in \bF_{3}^{3} \mid \mathrm{pr}_{i}(g_{i}) \neq h_{i} \textrm{ for any $1 \leq i \leq 3$}\}. 
 \end{align*}
Since the set $-g_{4} +  \{(h_{1}, h_{2}, h_{3}) \in \bF_{3}^{3} \mid \mathrm{pr}_{i}(g_{i}) \neq h_{i} \textrm{ for any $1 \leq i \leq 3$}\}$ contains a basis of $\bF_{3}^{3}$ and $\varphi_{4} \neq 0$, we have 
 \[
\{(h_{1}, h_{2}, h_{3}) \in \bF_{3}^{3} \mid \mathrm{pr}_{i}(g_{i}) \neq h_{i} \textrm{ for any $1 \leq i \leq 3$}\} \not\subset g_{4}\ker(\varphi_{4}), 
 \]
 which completes the proof. 
\end{proof}

The following is the result which corresponds to \cite[Proposition 3.6.1]{MRkoly}.

\begin{lemma}\label{lem:chev}
Let $c_{1}, c_{2}, c_{3}, c_{4} \in H^{1}(G_{\bQ}, \overline{T})$ be non-zero elements. 
Suppose that 
\[
\dim_{\bF_{3}} \left( \bF_{3}c_{1} + \bF_{3} c_{2} + \bF_{3} c_{3} + \bF_{3} c_{4} \right) \geq 3. 
\]
Then there are infinitely many primes $\ell \in \cP$ satisfying $\mathrm{loc}_{\ell}(c_{i}) \neq 0$ for any $1 \leq i \leq 4$. 
Here, $\mathrm{loc}_{\ell} \colon H^{1}(G_{\bQ}, \overline{T}) \longrightarrow H^{1}(G_{\bQ_{\ell}}, \overline{T})$ denotes the localization map at $\ell$. 
\end{lemma}

\begin{remark}
When $\bF = \bF_{3}$ and $\dim_{\bF_{3}} \left( \bF_{3}c_{1} + \bF_{3} c_{2} + \bF_{3} c_{3} + \bF_{3} c_{4} \right) = 2$, the conclusion of Lemma \ref{lem:chev} is not valid. 
In fact, if $c_{3} = c_{1} + c_{2}$ and $c_{4} = c_{1} - c_{2}$, then one of the elements $\mathrm{loc}_{\ell}(c_{1})$, $\mathrm{loc}_{\ell}(c_{2})$, $\mathrm{loc}_{\ell}(c_{3})$, and $\mathrm{loc}_{\ell}(c_{4})$ are zero for all but finitely many primes $\ell \in \cP$ since $H^{1}_{\mathrm{ur}}(G_{\bQ_{\ell}}, \overline{T}) \cong \bF_{3}$. 
\end{remark}

\begin{remark}
Lemma \ref{lem:chev} is only used for proving Lemma \ref{lem:change}. 
\end{remark}

\begin{proof}
The proof of this lemma is based on that of \cite[Proposition 3.6.1]{MRkoly}. 
Fix an element $\tau \in G_{\bQ(\mu_{3^{n}})}$ such that $T/(\tau - 1)T \cong R$. 
Put $F := \bQ(\mu_{3^{n}}, T)$. 
Since we assume that 
\[
H^{1}(\Gal(F/\bQ), \overline{T}) = 0, 
\]
the restriction map induces  an injection  
\[
H^{1}(G_{\bQ}, \overline{T}) \longhookrightarrow H^{1}(G_{F}, \overline{T})^{G_{\bQ}} = \Hom(G_{F}, \overline{T})^{G_{\bQ}}. 
\]
Since $\overline{T}$ is an irreducible $G_{\bQ}$-module, the map 
\begin{align}\label{map:inj}
\Hom(G_{F}, \overline{T})^{G_{\bQ}} \longhookrightarrow \Hom(G_{F}, \overline{T}/(\tau - 1)\overline{T})
\end{align}
is injective. Let $\overline{c}_{i} \in \Hom(G_{F}, \overline{T}/(\tau - 1)\overline{T})$ denote the image of $c_{i}$ under the injection \eqref{map:inj}. 
We also put 
\[
H_{i} := \{g \in G_{F} \mid c_{i}(\tau g) = 0 \textrm{ in } \overline{T}/(\tau-1)\overline{T}\}. 
\]
As mentioned in the proof of \cite[Proposition 3.6.1]{MRkoly}, the value  $c_{i}(\tau g) \bmod{(\tau-1)\overline{T}}$ is well-defined since $g \in G_{F}$ acts trivially on $\overline{T}$. 
Note that $\overline{c}_{i}$ is surjective since $\overline{c}_{i} \neq 0$. 
Hence we see that there is an element $g_{i} \in G_{F}$ such that $H_{i} = g_{i} \ker(\overline{c}_{i})$. 
Since the map \eqref{map:inj} is injective, we have 
$\dim_{\bF_{3}} \left( \bF_{3} \overline{c}_{1} + \bF_{3} \overline{c}_{2} + \bF_{3} \overline{c}_{3} + \bF_{3} \overline{c}_{4} \right) \geq 3$ by assumption. 
Hence Lemma \ref{lem:coset} shows that there is an element $g \in G_{F} \setminus (H_{1} \cup H_{2} \cup H_{3} \cup H_{4})$. 

For each $1 \leq i \leq 4$, we put $F_{i} := \overline{F}^{\ker(c_{i})}$. Note that $F/\bQ$ is a Galois extension since $c_{i} \in \Hom(G_{F}, \overline{T})^{G_{\bQ}}$. 
Let $S$ be the set of rational primes whose Frobenius conjugacy class in $\Gal(F_{1}F_{2}F_{3}F_{4}/\bQ)$ is the class of $\tau g$. 
Note that for any prime $\ell \in S$, we have 
\[
H^{1}_{\mathrm{ur}}(G_{\bQ_{\ell}}, \overline{T}) \cong \overline{T}/(\mathrm{Fr}_{\ell}-1)\overline{T} = \overline{T}/(\tau-1)\overline{T} \cong \bF. 
\]
Hence $S$ is an infinite set and $\mathrm{loc}_{\ell}(c_{i}) \neq 0$ for any $1 \leq i \leq 4$ and $\ell \in S$. 
Since the image of $\tau g$ in $\Gal(\bQ(\mu_{p^{n}})/\bQ)$ is trivial, we have $\ell \equiv 1 \bmod{ p^{n} }$, and so $S \subset \cP$. 
\end{proof}

\begin{corollary}\label{cor:chev}
Let $c_{1}, c_{2}, c_{3} \in H^{1}(G_{\bQ}, \overline{T})$ be non-zero elements. 
Then there are infinitely many primes $\ell \in \cP$ satisfying $\mathrm{loc}_{\ell}(c_{i}) \neq 0$ for any $1 \leq i \leq 3$. 
\end{corollary}
\begin{proof}
Note that $\dim_{\bF_{3}}(H^{1}(G_{\bQ}, \overline{T})) = \infty$. 
When $\dim_{\bF_{3}} \left( \bF_{3}c_{1} + \bF_{3} c_{2} + \bF_{3} c_{3} \right) \geq 2$, there exists an element $c \in H^{1}(G_{\bQ}, \overline{T})$ satisfying 
\[
\dim_{\bF_{3}} \left( \bF_{3}c_{1} + \bF_{3} c_{2} + \bF_{3} c_{3} + \bF_{3} c \right) \geq 3. 
\]
Hence this corollary follows from Lemma \ref{lem:chev}. 
When $\dim_{\bF_{3}} \left( \bF_{3}c_{1} + \bF_{3} c_{2} + \bF_{3} c_{3} \right) = 1$, we may assume that $c_{1} = c_{2} = c_{3}$. 
Then the same argument shows that there are infinitely many primes $\ell \in \cP$ satisfying $\mathrm{loc}_{\ell}(c_{i}) \neq 0$ for any $1 \leq i \leq 3$. 
\end{proof}

\subsection{Connectedness of the graph $\cX^0$}
Let $\cG$ be a Selmer structure on $T$. We denote by $\overline{\cG}$ the Selmer structure on $\overline{T}$ induced by $\cG$, that is, 
\[
H^{1}_{\overline{\cG}}(G_{\bQ_{\ell}}, \overline{T}) := \mathrm{im}\left( H^{1}_{\cG}(G_{\bQ_{\ell}}, T) \longrightarrow H^{1}(G_{\bQ_{\ell}}, \overline{T})\right)
\]
for any rational prime $\ell$. 
Since we assume that $\overline{T}$ is self-dual, one can regard $\overline{\cG}^{*}$ as a Selmer structure on $\overline{T}$. 
Suppose that 
\begin{itemize}
\item 
$\cG$ is cartesian and  residually self-dual (i.e., $\overline{\cG} = \overline{\cG}^{*}$). 
\end{itemize} 
Note that residual self-duality implies that $\chi(\cG) = 0$. 
In this subsection, we fix a rational prime $r$ such that 
\begin{itemize}
\item $\dim_{\bF}(H^{1}(G_{\bQ_{r}}, \overline{T})/H^{1}_{\overline{\cG}}(G_{\bQ_{r}}, \overline{T})) = 1$,   
\item $H^{1}(G_{\bQ_{r}}, T) \longrightarrow H^{1}(G_{\bQ_{r}}, \overline{T})$ is surjective. 
\end{itemize}
We put $\cF := \cG^{r}$. 

\begin{remark}
When $T = \mathrm{Ind}_{G_{\bQ_{n}}}^{G_{\bQ}}(E[p^{m}])$, 
$\cG = \cF_{\mathrm{cl}}$ and $r = 3$, all assumptions in this appendix are satisfied and we have $\cF = \cF_{\mathrm{can}}$.  
\end{remark}

We  set $\cP(\cG, r) = \cP \setminus (S(\cG) \cup \{r\})$ and $\cN(\cG, r)$ denotes the set of square products in $\cP(\cG, r)$.  
For notational simplicity, we also write $\cF$ for the Selmer structure on $\overline{T}$ induced by $\cF$. 
For any square-free integer $d$, we define 
\begin{align*}
\lambda(d) &:= \dim_{\bF}(H^{1}_{\cF(d)}(G_{\bQ}, \overline{T})), 
\\
\lambda^{*}(d) &:= \dim_{\bF}(H^{1}_{\cF^{*}(d)}(G_{\bQ}, \overline{T})).  
\end{align*}
Following \cite[Definition 4.3.6]{MRkoly}, we define the graph $\cX^{0} := \cX^0(\cF)$ as follows.
\begin{itemize}
\item The vertices of $\cX^0$ are integers $d \in \cN(\cG,r)$ with $\lambda^{*}(d) = 0$. 
\item For any vertices $d, d\ell \in \cX^{0}$ with $\ell \in \cP(\cG,r)$, we join $d$ and $d\ell$ by an edge in $\cX^0$
if and only if $H^1_{\cF(d)}(G_{\bQ}, \overline{T}) \neq  H^1_{\cF_{\ell}(d)}(G_{\bQ}, \overline{T})$. 
\end{itemize}
In this subsection, we prove the connectedness of the graph $\cX^0$ which is one of the most important facts in the theory of Kolyvagin systems.  

\begin{lemma}
The Selmer structure $\cF$ is cartesian and $\chi(\cF) = 1$. 
\end{lemma}
\begin{proof}
Since $H^{1}(G_{\bQ_{r}}, T) \longrightarrow H^{1}(G_{\bQ_{r}}, \overline{T})$ is surjective and $\cG$ is cartesian, we see that $\cF = \cG^{r}$ is cartesian. 
Applying Theorem \ref{pt} with $\cF_{1} = \overline{\cG}$ and $\cF_{2} = \cF$, we obtain 
\[
\chi(\cF) = \chi(\cG) + \dim_{\bF}(H^{1}(G_{\bQ_{r}}, \overline{T})/H^{1}_{\overline{\cG}}(G_{\bQ_{r}}, \overline{T})) = 1. 
\]
\end{proof}

The following  lemma is an  applications of Theorem \ref{pt}. 

\begin{lemma}\label{lem:core1}
Let $d \in \cN(\cG,r)$ be an integer. Then the following claims are valid. 
\begin{itemize}
\item[(1)] $\lambda(d) = \lambda^{*}(d) + 1$. 
\item[(2)] $|\lambda(d) - \lambda(d\ell)| \leq 1$ for any prime $\ell \in \cP(\cG,r)$ with $\ell \nmid d$. 
\item[(3)] $|\lambda^{*}(d) - \lambda^{*}(d\ell)| \leq 1$  for any prime $\ell \in \cP(\cG,r)$ with $\ell \nmid d$. 
\item[(4)] If $H^1_{\cF(d)}(G_{\bQ}, \overline{T}) \neq H^1_{\cF_{\ell}(d)}(G_{\bQ}, \overline{T})$, then 
$\lambda^{*}(d\ell) \leq \lambda^{*}(d) $. 
\item[(5)] If $H^1_{\cF^{*}(d)}(G_{\bQ}, \overline{T}) \neq H^1_{\cF^{*}_{\ell}(d)}(G_{\bQ}, \overline{T})$, then 
$\lambda(d\ell) = \lambda(d) - 1$ and $\lambda^{*}(d\ell) = \lambda^{*}(d)  -  1$. 
\end{itemize}
In particular, $\nu(d) \geq \lambda^{*}(1)$ for any integer $d \in \cN(\cG,r)$ with $\lambda^{*}(d) = 0$. 
\end{lemma}
\begin{proof}
Claim (1) follows from \cite[Proposition 4.1.4]{MRkoly} and the fact that $\lambda(1) - \lambda^{*}(1) = \chi(\cF) = 1$. 
Claims (2) and (3) follow from \cite[Lemma 4.1.7(i)]{MRkoly}. 

Suppose that $H^1_{\cF(d)}(G_{\bQ}, \overline{T}) \neq H^1_{\cF_{\ell}(d)}(G_{\bQ}, \overline{T})$. 
Since $H^{1}_{\mathrm{ur}}(G_{\bQ_{\ell}}, \overline{T}) \cong \bF$, 
applying Theorem \ref{pt} with $\cF_{1} = \cF_{\ell}(d)$ and $\cF_{2} = \cF(d)$, we see that 
$H^1_{\cF^{*}(d)}(G_{\bQ}, \overline{T}) = H^1_{(\cF^{*})^{\ell}(d)}(G_{\bQ}, \overline{T}) \supset H^1_{\cF^{*}(d\ell)}(G_{\bQ}, \overline{T})$, which implies claim (4). 

Since $\cF^{*} \subset \cF$ by definition, if $H^1_{\cF^{*}(d)}(G_{\bQ}, \overline{T}) \neq H^1_{\cF^{*}_{\ell}(d)}(G_{\bQ}, \overline{T})$, then we have $H^1_{\cF(d)}(G_{\bQ}, \overline{T}) \neq H^1_{\cF_{\ell}(d)}(G_{\bQ}, \overline{T})$. Hence claim (5)  follows from \cite[Lemma 4.1.7(iv)]{MRkoly}
\end{proof}


\begin{lemma}\label{lem:koly-join1}
For any vertices $d, d\ell\in \cX^{0}$ with $\ell \in \cP(\cG,r)$,  there is a path in $\cX^{0}$ from $d$ to $d\ell$. 
\end{lemma}
\begin{proof}
This lemma is proved by Mazur and Rubin in {\cite[Lemma 4.3.9]{MRkoly}}. 
Note that \cite[Proposition 3.6.1]{MRkoly} is used in the proof of \cite[Lemma 4.3.9]{MRkoly}. 
However, exactly the same argument as in \cite[Lemma 4.3.9]{MRkoly} works even if we use Corollary \ref{cor:chev} instead of \cite[Proposition 3.6.1]{MRkoly}. 
\end{proof}

\begin{lemma}\label{lem:change}
For each integer  $1 \leq i \leq 2$, let $d_{i} \in \cX^{0}$  and $\ell_{i} \in \cP(\cG,r)$ with $\ell_{i} \mid d_{i}$. 
Suppose that $\nu(d_{1}) = \nu(d_{2}) = \lambda^{*}(1)$ and $\ell_{1} \neq \ell_{2}$. 
Then there exists a prime $q \in \cP(\cG,r)$ with $q \nmid d_{1}d_{2}$ such that 
there is a path in $\cX^{0}$ from $d_{i}$ to $d_{i}q/\ell_{i}$  for each integer $1 \leq i \leq 2$.  
\end{lemma}
\begin{proof}
Let $1 \leq i \leq 2$ and put $e_{i} = d_{i}/\ell_{i}$.
Since $\nu(e_{i}) = \lambda^{*}(1) - 1$,  we have $\lambda(e_{i}) = 2$ and $\lambda^{*}(e_{i}) = 1$ by Lemma \ref{lem:core1}.  
By definition, we have 
\[
H^{1}_{\cF^{*}(e_{i})}(G_{\bQ}, \overline{T}) = H^{1}_{\overline{\cG}_{r}(e_{i})}(G_{\bQ}, \overline{T}) \subset 
 H^{1}_{\cF(e_{i})}(G_{\bQ}, \overline{T}). 
\]
Moreover, since $\lambda^{*}(d_{i}) = 0$, we also have 
\[
H^{1}_{\cF^{*}(e_{i})}(G_{\bQ}, \overline{T})  \cap H^{1}_{\cF(d_{i})}(G_{\bQ}, \overline{T}) \subset H^{1}_{\cF^{*}(d_{i})}(G_{\bQ}, \overline{T}) = 0. 
\] 
Since $\lambda(e_{i}) = 2$ and $\lambda(d_{i}) = \lambda^{*}(e_{i}) = 1$, we obtain a decomposition 
\[
H^{1}_{\cF(e_{i})}(G_{\bQ}, \overline{T}) = H^{1}_{\cF(d_{i})}(G_{\bQ}, \overline{T}) \oplus H^{1}_{\cF^{*}(e_{i})}(G_{\bQ}, \overline{T}). 
\]
Take non-zero elements $c^{(i)}_{1} \in H^{1}_{\cF(d_{i})}(G_{\bQ}, \overline{T})$ and $c^{(i)}_{2} \in H^{1}_{\cF^{*}(e_{i})}(G_{\bQ}, \overline{T})$. 
By definition, we have $H^{1}_{\cF(e_{i})}(G_{\bQ}, \overline{T}) \cap \ker(\mathrm{loc}_{r}) = H^{1}_{\cF^{*}(e_{i})}(G_{\bQ}, \overline{T})$. 
Hence we see that $\mathrm{loc}_{r}(c^{(1)}_{1}) \neq 0 \neq \mathrm{loc}_{r}(c^{(2)}_{1})$ and $\mathrm{loc}_{r}(c^{(1)}_{2}) = 0 = \mathrm{loc}_{r}(c^{(2)}_{2})$. 

Let us show that there is a prime $q \in \cP(\cG,r)$ such that $\mathrm{loc}_{q}(c^{(i)}_{j}) \neq 0$ for any $i,j \in \{1,2\}$. 
If $c^{(2)}_{2} \not\in H^{1}_{\cF(e_{1})}(G_{\bQ}, \overline{T})$, then this claim follows from Lemma \ref{lem:chev}. 
Suppose that $c^{(2)}_{2} \in H^{1}_{\cF(e_{1})}(G_{\bQ}, \overline{T})$, that is, $c^{(2)}_{2} = a c^{(1)}_{1} + b c^{(1)}_{2}$ for some $a,b \in \bF$. 
Then 
\[
0 = \mathrm{loc}_{q}(c^{(2)}_{2}) =  \mathrm{loc}_{q}(ac^{(1)}_{1}) +  \mathrm{loc}_{q}(bc^{(1)}_{2}) =  \mathrm{loc}_{q}(ac^{(1)}_{1}). 
\]
Since $\mathrm{loc}_{q}(c^{(1)}_{1}) \neq 0$, we may assume that $c^{(1)}_{2} = c^{(2)}_{2}$. 
Then Corollary \ref{cor:chev} shows that 
there is a prime $q \in \cP(\cG,r)$ such that $\mathrm{loc}_{q}(c^{(i)}_{j}) \neq 0$ for any $i,j \in \{1,2\}$. 

Let us prove that $q$ is a desired prime. Lemma \ref{lem:core1} and the fact that $\mathrm{loc}_{q}(c^{(i)}_{1}) \neq 0$ imply  
$\lambda^{*}(d_{i}q) \leq \lambda^{*}(d_{i}) = 0$, that is, $d_{i}q \in \cX^{0}$. 
Since $\mathrm{loc}_{q}(c^{(i)}_{2}) \neq 0$, we have 
\[
H^{1}_{\cF_{q}^{*}(e_{i})}(G_{\bQ}, \overline{T}) \neq H^{1}_{\cF^{*}(e_{i})}(G_{\bQ}, \overline{T}). 
\]
Hence Lemma \ref{lem:core1} shows that  $\lambda^{*}(e_{i}q) = \lambda^{*}(e_{i}) -1 = 0$, 
that is, $e_{i}q \in \cX^{0}$. 
Since $d_{i}, d_{i}q, e_{i}q \in \cX^{0}$, Lemma \ref{lem:koly-join1} shows that there is a path in $\cX^{0}$ from  $d_{i}$ to $e_{i}q$. 
\end{proof}

\begin{corollary}[{\cite[Proposition 4.3.11]{MRkoly}}]\label{cor:koly-join3}
For any vertices $d_{1}, d_{2} \in \cX^{0}$ satisfying $\nu(d_{1}) = \nu(d_{2}) = \lambda^{*}(1)$,  
there is a path in $\cX^0$ from $d_{1}$ to $d_{2}$. 
\end{corollary}
\begin{proof}
Put $d := \mathrm{gcd}(d_{1}, d_{2})$. 
Let us show this corollary by induction on $\lambda^{*}(1) - \nu(d)$. 
When $ \nu(d) = \lambda^{*}(1)$, then $d_{1} = d_{2}$, and there is nothing to prove. 
When $\nu(d) < \lambda^{*}(1)$, there are primes $\ell_{1}, \ell_{2} \in \cP(\cG,r)$ with $\ell_{1} \mid d_{1}/d$ and $\ell_{2} \mid d_{2}/d$. 
Then by Lemma \ref{lem:change}, we have a prime $q \in \cP(\cG,r)$ with $q \nmid d_{1}d_{2}$ such that $d_{1}q/\ell_{1}, d_{2}q/\ell_{2} \in \cX^{0}$ and 
that there is a path in $\cX^{0}$ from  $d_{i}$ to $d_{i}q/\ell_{i}$  for any $1 \leq i \leq 2$. 
Since $\nu(d) < \nu(\mathrm{gcd}(d_{1}q/\ell_{1}, d_{2}q/\ell_{2}))$, the induction hypothesis shows that there is a path in $\cX^0$ from $d_{1}q/\ell_{1}$ to $d_{2}q/\ell_{2}$, and hence we obtain  a path in $\cX^0$ from $d_{1}$ to $d_{2}$. 
\end{proof}

\begin{lemma}\label{lem:koly-join2}
For any vertex  $d \in \cX^{0}$ with $\nu(d) > \lambda^*(1)$, 
there is a vertex $e \in \cX^{0}$ with $\nu(e) < \nu(d)$ such that there is  a path in $\cX^{0}$ from $d$ to $e$. 
\end{lemma}
\begin{proof}
Exactly the same argument as in \cite[Proposition 4.3.10]{MRkoly} works even if we use Corollary \ref{cor:chev} instead of \cite[Proposition 3.6.1]{MRkoly}. 
Hence this lemma is proved by Mazur and Rubin in \cite[Proposition 4.3.10]{MRkoly}. 
\end{proof}

Since $\lambda^{*}(1) \leq \nu(d)$ for any vertex $d \in \cX^{0}$, Corollary \ref{cor:koly-join3} and Lemma \ref{lem:koly-join2} imply the following 

\begin{theorem}[{\cite[Theorem 4.3.12]{MRkoly}}]\label{thm:connected}
The graph $\cX^0$ is connected.
\end{theorem}

\subsection{Kolyvagin systems}

We use the same notations as in the previous subsection. 
In this subsection, we prove  Theorem \ref{thm:koly0} and Proposition \ref{prop:reg-isom} when $p=3$. 

\begin{definition}
Let $\mathrm{KS}_{1}(T, \cF)$ denote the module of Kolyvagin systems of rank $1$ (for $\cF$), that is, the set of elements in $\prod_{d\in\cN(\cG,r)}H^{1}_{\cF(d)}(G_{\bQ}, T) \otimes_{\bZ} G_{d}$ satisfying the finite-singular relations. 
\end{definition}

\begin{proposition}\label{prop:inj,p=3}
For any integer $d \in \cN(\cG,r)$ with $\lambda^{*}(d) = 0$, the canonical projection 
\[
\mathrm{KS}_{1}(T, \cF) \longrightarrow H^{1}_{\cF(d)}(G_{\bQ}, T) \otimes_{\bZ} G_{d}
\]
is injective. 
\end{proposition}
\begin{proof}
Let $\fm_{R}$ denote the maximal ideal of $R$. Since $\cF$ is cartesian, so is $\cF(d)$ for any integer $d \in \cN(\cG,r)$ (see \cite[Corollary 3.18]{sakamoto}). 
Hence, by \cite[Lemma 3.13]{sakamoto}, the canonical injection $\overline{T} \longhookrightarrow R$ induces an isomorphism 
$H^{1}_{\cF(d)}(G_{\bQ}, \overline{T}) \stackrel{\sim}{\longrightarrow} H^{1}_{\cF(d)}(G_{\bQ}, T)[\fm_{R}]$ for any integer $d \in \cN(\cG,r)$. 
Therefore, we may assume that $R = \bF$ and $T = \overline{T}$ since $\mathrm{KS}_{1}(\overline{T}, \cF) \stackrel{\sim}{\longrightarrow} \mathrm{KS}_{1}(T, \cF)[\fm_{R}]$. 

Take an integer $d \in \cN(\cG,r)$ with $\lambda^{*}(d) = 0$. 
Let  $(\kappa_{e})_{e \in \cN(\cG,r)} \in \mathrm{KS}_{1}(T, \cF)$ be a Kolyvagin system satisfying $\kappa_{d} = 0$. 
Let us show $\kappa_{e} = 0$ by induction on $\lambda^{*}(e)$. 
When $\lambda^{*}(e) = 0$, there is a path in $\cX^{0}$ from $d$ to $e$ by Theorem \ref{thm:connected}. Hence the finite-singular relation and \cite[Lemma 4.3.8]{MRkoly} imply $\kappa_{e} = 0$. 
Suppose that $\lambda^{*}(e) > 0$, and take a non-zero element $c \in H^{1}_{\cF^{*}(e)}(G_{\bQ}, E[p])$. 
If $\kappa_{e} \neq 0$, then by Corollary \ref{cor:chev}, there is a prime $\ell \in \cP(\cG,r)$ with $\ell \nmid e$ such that 
$\mathrm{loc}_{\ell}(\kappa_{e}) \neq 0$ and $\mathrm{loc}_{\ell}(c) \neq 0$. 
Since $\mathrm{loc}_{\ell}(c) \neq 0$,  we have $\lambda^{*}(e\ell) = \lambda^{*}(e)-1$ by Lemma \ref{lem:core1}. 
Hence the induction hypothesis and the finite-singular relation imply 
\[
0 \neq \varphi^{\mathrm{fs}}_{\ell}(\kappa_{e}) = v_{\ell}(\kappa_{e \ell}) = 0. 
\]
Therefore, we conclude that $\kappa_{e} = 0$. 
\end{proof}

As explained in \S\ref{sec:rank reduction}, for any integer $d \in \cN(\cG,r)$, the exact sequence 
\[
0 \longrightarrow H^{1}_{\cF(d)}(G_{\bQ}, T) 
\longrightarrow H^{1}_{\cF^{d}}(G_{\bQ}, T) 
\longrightarrow \bigoplus_{\ell \mid d}H^{1}_{/{\rm tr}}(G_{\bQ_{\ell}}, T)
\]  
induces a natural homomorphism $\Pi_{d} \colon X_{d}^{1}(T, \cF) \longrightarrow H^{1}_{\cF(d)}(G_{\bQ}, T) \otimes_{\bZ} G_{d}$, 
and we obtain 
\[
\mathrm{Reg}_{1} \colon {\rm SS}_{1}(T, \cF) \longrightarrow {\rm KS}_{1}(T, \cF). 
\]  
By construction, the following diagram commutes: 
\begin{align}
\begin{split}\label{diag:reg}
\xymatrix{
{\rm SS}_{1}(T, \cF) \ar[d]^-{\mathrm{Reg}_{1}} \ar[r] & X_{d}^{1}(T, \cF) \ar[d]^-{(-1)^{\nu(d)}\Pi_{d}}
\\
{\rm KS}_{1}(T, \cF) \ar[r] & H^{1}_{\cF(d)}(G_{\bQ}, T) \otimes_{\bZ} G_{d}
}
\end{split}
\end{align}

\begin{theorem}\label{thm:reg,p=3}
Let $\cG$ be a residually self-dual cartesian Selmer structure on $T$ and let $r$ be a rational prime satisfying 
\begin{itemize}
\item $\dim_{\bF}(H^{1}(G_{\bQ_{r}}, \overline{T})/H^{1}_{\overline{\cG}}(G_{\bQ_{r}}, \overline{T})) = 1$,   
\item $H^{1}(G_{\bQ_{r}}, T) \longrightarrow H^{1}(G_{\bQ_{r}}, \overline{T})$ is surjective. 
\end{itemize}
We set $\cF := \cG^{r}$. 
Then the map $\mathrm{Reg}_{1} \colon {\rm SS}_{1}(T, \cF) \longrightarrow {\rm KS}_{1}(T, \cF)$ is an isomorphism. 
\end{theorem}
\begin{remark}
When $\cG = \cF_{\mathrm{cl}}$ and $r = p = 3$, we have $\cF = \cF_{\mathrm{can}}$. Hence Theorem \ref{thm:reg,p=3} shows that Proposition \ref{prop:reg-isom} is valid when $p=3$.  
\end{remark}
\begin{proof}
Let $d \in \cN(\cG,r)$ be an integer with $\lambda^{*}(d) = 0$. 
Then, by \cite[Lemma 4.6]{sakamoto}, we have 
\[
H^{1}_{\cF(d)}(G_{\bQ}, T)  \cong R \,\,\, \textrm{ and } \,\,\, H^{1}_{\cF^{d}}(G_{\bQ}, T) \cong R^{1+\nu(d)}. 
\]
Moreover, by Theorem \ref{pt}, we have a split exact sequence of free $R$-modules: 
\[
0 \longrightarrow H^{1}_{\cF(d)}(G_{\bQ}, T) 
\longrightarrow H^{1}_{\cF^{d}}(G_{\bQ}, T) 
\longrightarrow \bigoplus_{\ell \mid d}H^{1}_{/{\rm tr}}(G_{\bQ_{\ell}}, T) 
\longrightarrow 0. 
\]
These facts shows that $\Pi_{d}$ is an isomorphism. 
By \cite[Theorem 4.7]{sakamoto}, the projection map 
\[
{\rm SS}_{1}(T, \cF)  \longrightarrow X_{d}^{1}(T, \cF)
\] 
is also an isomorphism. 
Hence this theorem follows from Proposition \ref{prop:inj,p=3} and the commutative diagram \eqref{diag:reg}. 
\end{proof}

Next, we prove Theorem \ref{thm:koly0} when $p=3$. 
First, let us show that the regulator map (constructed in \cite[\S5.2]{sakamoto-koly0})
\[
\mathrm{Reg}_{0} \colon \mathrm{SS}_{0}(T, \cG) \longrightarrow \mathrm{KS}_{0}(T, \cG)
\]
is an isomorphism. Recall that we fix an isomorphism $H^{1}_{/\mathrm{ur}}(G_{\bQ_{\ell}}, T) \cong R$ for any prime $\ell \in \cP$ in order to define Kolyvagin systems of rank $0$. 

Suppose that $r \in \cP \setminus S(\cG)$. Note that for any prime $r \in \cP \setminus S(\cG)$, we have  
\begin{itemize}
\item $\dim_{\bF}(H^{1}(G_{\bQ_{r}}, \overline{T})/H^{1}_{\overline{\cG}}(G_{\bQ_{r}}, \overline{T})) = \dim_{\bF}(H^{1}_{/\mathrm{ur}}(G_{\bQ_{r}}, \overline{T}))  =1$,   
\item $H^{1}(G_{\bQ_{r}}, T) \longrightarrow H^{1}(G_{\bQ_{r}}, \overline{T})$ is surjective. 
\end{itemize}
The fixed isomorphism $H^{1}_{/\mathrm{ur}}(G_{\bQ_{r}}, T) \cong R$ induces an isomorphism $W_{d} \cong W_{dr}$ for any  integer $d \in \cN(\cG, r)$  (see Definition \ref{def:stark}).  
Hence we obtain an isomorphism $X_{dr}^{0}(T, \cG) \cong X_{d}^{1}(T, \cF)$ for any integer $d \in \cN(\cG, r)$, and it naturally induces an isomorphism 
\[
\mathrm{SS}_{0}(T, \cG) \stackrel{\sim}{\longrightarrow} \mathrm{SS}_{1}(T, \cF). 
\]
By the definition of Kolyvagin system of rank $0$, we have a homomorphism 
\[
\mathrm{KS}_{0}(T, \cG) \longrightarrow \mathrm{KS}_{1}(T, \cF); \, (\kappa_{d,\ell})_{(d,\ell) \in \cM(\cG)} \mapsto (\kappa_{d, r})_{d \in \cN(\cG,r)}. 
\]
Here $\cM(\cG) := \bigcup_{q \in \cP \setminus S(\cG)}\cN(\cG,q) \times \{q\}$. 
By \cite[Lemma 5.4]{sakamoto-koly0}, we have the following commutative diagram: 
\begin{align}
\begin{split}\label{diag:comm-reg0}
\xymatrix{
\mathrm{SS}_{0}(T, \cG) \ar[r]^-{\cong} \ar[d]^-{\mathrm{Reg}_{0}} & \mathrm{SS}_{1}(T, \cF) \ar[d]^-{\mathrm{Reg}_{1}}
\\
\mathrm{KS}_{0}(T, \cG)  \ar[r] & \mathrm{KS}_{1}(T, \cF). 
}
\end{split}
\end{align}

\begin{proposition}
For any residually self-dual cartesian Selmer structure $\cG$ on $T$, 
the map $\mathrm{Reg}_{0}$ is an isomorphism. 
\end{proposition}
\begin{proof}
Theorem \ref{thm:reg,p=3} shows that the homomorphism $\mathrm{Reg}_{1}$ in the commutative diagram \eqref{diag:comm-reg0} is an isomorphism. 
Hence, it suffices to show that the map $\mathrm{KS}_{0}(T, \cG)  \longrightarrow  \mathrm{KS}_{1}(T, \cF)$ is injective. 

Let $(\kappa_{d,\ell})_{(d,\ell) \in \cM(\cG)} \in \mathrm{KS}_{0}(T, \cG)$ be a Kolyvagin system satisfying $\kappa_{d,r} = 0$ for any  $d \in \cN(\cG,r)$. 
Take a prime $q \in \cP \setminus S(\cG)$ and an integer $e \in \cN(\cG,q) \cap \cN(\cG,r)$ with $H^{1}_{\overline{\cG}(eq)}(G_{\bQ}, \overline{T}) = 0$. 
Since $\overline{T} = \overline{T}^{\vee}(1)$ and $\overline{\cG}(eq) = \overline{\cG}^{*}(eq)$, by \cite[Lemmas 3.13 and 3.14]{sakamoto}, we have isomorphisms 
\[
H^{1}_{\cG(eq)}(G_{\bQ}, T)[\fm_{R}] \cong H^{1}_{\overline{\cG}(eq)}(G_{\bQ}, \overline{T})  \cong H^{1}_{\cG^{*}(eq)}(G_{\bQ}, T^{\vee}(1))[\fm_{R}]. 
\]
Hence $H^{1}_{\cG(eq)}(G_{\bQ}, T) = H^{1}_{\cG^{*}(eq)}(G_{\bQ}, T) = 0$. 
Applying Theorem \ref{pt} with $\cF_{1} = \cG(e)$ and $\cF_{2} = \cG^{q}(e)$, we obtain an isomorphism 
\[
\varphi_{q}^{\mathrm{fs}} \colon H^{1}_{\cG^{q}(e)}(G_{\bQ}, T) \stackrel{\sim}{\longrightarrow} H^{1}_{/\mathrm{tr}}(G_{\bQ_{q}}, T) \cong R \otimes_{\bZ}G_{q}. 
\]
The definition of Kolyvagin system of rank $0$ implies that 
\[
\varphi_{q}^{\mathrm{fs}}(\kappa_{e,q}) = -v_{r}(\kappa_{e,r}) = 0, 
\]
and hence we have $\kappa_{e,q} = 0$. 
By proposition \ref{prop:inj,p=3}, the map $\mathrm{KS}_{1}(T, \cG^{q}) \longrightarrow H^{1}_{\cG^{q}(e)}(G_{\bQ}, T)$ 
is injective. Therefore, $0 = (\kappa_{d,q})_{d \in \cN(\cG,q)} \in \mathrm{KS}_{1}(T, \cG^{q})$. 
Since $q$ is an arbitrary prime and $\cM(\cG) = \bigcup_{q \in \cP \setminus S(\cG)}\cN(\cG,q) \times \{q\}$, we have $0 = (\kappa_{d,\ell})_{(d,\ell)\in\cM(\cG)} \in \mathrm{KS}_{0}(T, \cG)$. 
\end{proof}

The same argument as in the proof of \cite[Theorem 5.8]{sakamoto-koly0} shows the following Theorem. 

\begin{theorem}
Let $\cG$ be a residually self-dual cartesian Selmer structure on the residually self-dual Galois representation $T$.  
\begin{itemize}
\item[(1)] For any element $(d,\ell) \in \cM(\cG)$ with $H^{1}_{\overline{\cG}_{\ell}(d)}(G_{\bQ}, \overline{T})=0$, the projection map 
\[
{\rm KS}_{0}(T,\cG) \longrightarrow  H^{1}_{\cG^{\ell}(d)}(G_{\bQ},T) \otimes_{\bZ} G_{d}
\]
is an isomorphism. 
In particular, the $R$-module ${\rm KS}_{0}(T,\cG)$ is free of rank $1$. 
\item[(2)] For any basis $\kappa \in {\rm KS}_{0}(T,\cG)$ and any integer $d \in \bigcup_{q \in \cP \setminus S(\cG)}\cN(\cG,q)$, we have 
\[
R \cdot \delta(\kappa)_{d} = \mathrm{Fitt}_{R}^{0}(H^{1}_{\cG^{*}(d)}(G_{\bQ}, T^{\vee}(1))^{\vee}). 
\]
\end{itemize}
\end{theorem}

\end{document}